\documentclass[12pt,a4paper,oneside]{amsart}
\DeclareSymbolFont{greekletters}{OML}{cmr}{m}{it}
\DeclareMathSymbol{\varrho}{\mathalpha}{greekletters}{"25}
\usepackage{amsmath,amssymb,amsthm,hyperref}
\usepackage{enumerate}
\usepackage{amsfonts}
\usepackage{amsthm}
\usepackage{amsmath}
\usepackage{amssymb}
\usepackage{t1enc}
\usepackage{comment}
\usepackage{graphicx}
\usepackage{mathrsfs}
\usepackage{mathptmx}
\usepackage{tikz}
\usepackage{version} % include, exclude and comment
\usepackage{color}
\usepackage{enumerate}
\usepackage{fancyhdr}
\usepackage{datetime}
\numberwithin{equation}{section}
\usepackage[bf]{caption}     % captionstyles
\usepackage{geometry}        % for good layout
\usepackage{bm}  % bold symbols in math mode \boldsymbol{\alpha}
\usepackage{buczosier}
\usepackage{mathtools}
\usepackage[utf8]{inputenc}

\newtheorem{theorem}{Theorem}[section]
\newtheorem{proposition}[theorem]{Proposition}
\newtheorem{lemma}[theorem]{Lemma}

\newtheorem{corollary}[theorem]{Corollary}
\newtheorem{claim}[theorem]{Claim}
\theoremstyle{definition}
\newtheorem{remark}[theorem]{Remark}
\newtheorem{definition}[theorem]{Definition}

\newcommand{\gggg}{\gamma}

\definecolor{aquam}{rgb}{0.5,1.0,1.0}
\definecolor{bbrown}{rgb}{0.75,0.38,0.15}
\definecolor{Cyan}{rgb}{0,0.6,0.6}
\definecolor{Darkblue}{rgb}{0,0,1}
\definecolor{Dodgerblue2}{rgb}{0,0.5,1}
\definecolor{Green}{rgb}{0,0.6,0.06}
\definecolor{Kahki}{rgb}{1,1,0.5}
\definecolor{Magenta}{rgb}{0.7,0,0.7}
\definecolor{bMagenta}{rgb}{1,.6,1}
\definecolor{Orange}{rgb}{0.8,0.3,0}
\definecolor{dOrchid}{rgb}{0.7,0.2,0.4}
\definecolor{Orchid}{rgb}{1,0.5,1}
\definecolor{Purple}{rgb}{0.65,0.07,0.85}
\definecolor{Royalblue}{rgb}{0.6,0.85,0.87}
\definecolor{Tan}{rgb}{0.54,0.42,0.23}
\definecolor{bTan}{rgb}{0.94,0.82,0.63}
\definecolor{zoltan}{rgb}{0,0.1,0.3}
\definecolor{Turquoise}{rgb}{0,0.85,0.87}
\definecolor{Yellow}{rgb}{1,1,0}
\definecolor{darkamber}{rgb}{0.4,0.19,0.28}
\definecolor{bYellow}{rgb}{1,1,0.6}
\definecolor{bRed}{rgb}{1,0.7,0.7}
\definecolor{boxcolb}{rgb}{0.87,0.77,0.75}%rosybrown
\definecolor{boxcol}{rgb}{0.6,0.85,0.87}%cadetblue
\definecolor{boxcolgreen}{rgb}{0.64,0.93,0.79}
\definecolor{boxcolaa}{rgb}{.75,.99,.70}
\definecolor{boxcolbb}{rgb}{0.39,0.50,0.56}
\definecolor{boxcolcc}{rgb}{1,0.81,0.65}
\definecolor{yy}{rgb}{0.43,0.21,.18}
\definecolor{gA}{gray}{0.5}
\definecolor{gB}{gray}{0.8}
\definecolor{gC}{gray}{0.9}
\usepackage{xcolor}

\hypersetup{
  colorlinks   = true, %Colours links instead of ugly boxes
  urlcolor     = blue, %Colour for external hyperlinks
  linkcolor    = blue, %Colour of internal links
  citecolor   = red %Colour of citations
}

\begin{document}

\title{ Generic H\"older level sets on Fractals }
\author{Zolt\'an Buczolich$^*$}
\address{Department of Analysis, ELTE E\"otv\"os Lor\'and\\
University, P\'azm\'any P\'eter S\'et\'any 1/c, 1117 Budapest, Hungary}
\email{zoltan.buczolich@ttk.elte.hu}
\urladdr{http://buczo.web.elte.hu, ORCID Id: 0000-0001-5481-8797}

\author{Bal\'azs Maga$^\text{\textdagger}$}
\address{Department of Analysis, ELTE E\"otv\"os Lor\'and\\
University, P\'azm\'any P\'eter S\'et\'any 1/c, 1117 Budapest, Hungary}
\email{mbalazs0701@gmail.com}
\urladdr{  http://magab.web.elte.hu/}

\author{ G\'asp\'ar V\'ertesy$^\text{\textdaggerdbl}$}
 \address{Alfr\'ed R\'enyi Institute of Mathematics, Re\'altanoda street 13-15, 1053 Budapest, Hungary}
\email{vertesy.gaspar@gmail.com}
\thanks{\scriptsize $^*$
The project leading to this application has received funding from the European Research Council (ERC) under the European Union’s Horizon 2020 research and innovation programme (grant agreement No. 741420).
This author was also supported by the Hungarian National Research, Development and Innovation Office--NKFIH, Grant 124003 and  at the time of completion of this paper was holding a visiting researcher position 
at the Rényi Institute.
}

\thanks{\scriptsize $^\text{\textdagger}$ This author was supported by the \'UNKP-21-3 New National Excellence of the Hungarian Ministry of Human Capacities, and by the Hungarian National Research, Development and Innovation Office-NKFIH, Grant 124749.}
\thanks{\scriptsize
$^\text{\textdaggerdbl}$  This author was supported by the \'UNKP-20-3 New National Excellence Program of the Ministry for Innovation and Technology from the source of the National Research, Development and Innovation Fund, and by the Hungarian National Research, Development and Innovation Office–NKFIH, Grant 124749.
}

\date{\today}

%\title{-}
%\author{Zolt\'an Buczolich, Bal\'azs Maga, G\'asp\'ar V\'ertesy}
%\date{}
\maketitle

\begin{abstract}
Hausdorff dimensions of level sets of generic continuous functions defined on fractals 
were considered  in two papers by R. Balka, Z. Buczolich and M. Elekes. In those papers  the  topological Hausdorff dimension of fractals was defined.
In this paper we start to study level sets of generic $1$-H\"older-$\aaa$ functions defined on fractals.  This is related to some sort of "thickness", "conductivity" properties of fractals.
The main concept of our paper is $D_{*}(\aaa, F)$ which  is the  essential supremum of the Hausdorff dimensions of the level sets of  a generic
$1$-H\"older-$\aaa$ function defined on the fractal $F$.
We prove some basic properties of $D_{*}(\aaa, F)$, we calculate its value for an example of a "thick fractal sponge", we show that for connected self similar sets $D_{*}(\aaa, F)$  it equals the Hausdorff dimension of almost
every level in the range of a generic $1$-H\"older-$\aaa$ function.
\end{abstract}

\setcounter{tocdepth}{3}
\tableofcontents

%%%%%%%%%%%%%%%%%%%%%%%%%%%%%%%%
%%%%%%%%%%%%%%%%%%%%%%%%%%%%%%%%
%%%%%%%%%%%%%%%%%%%%%%%%%%%%%%%%
%%%%%%%%%%%%%%%%%%%%%%%%%%%%%%%%
%%%%%%%%%%%%%%%%%%%%%%%%%%%%%%%%
%%%%%%%%%%%%%%%%%%%%%%%%%%%%%%%%

\section{Introduction}
  
In \cite{BBEtoph} the concept of topological Hausdorff dimension was introduced
(the definition of the topological Hausdorff dimension and the definition of some 
other concepts used in this introduction can be found in Section \ref{*secprel}).
The starting point of the research leading to the definition of topological Hausdorff dimension was a purely theoretic question concerning 
the Hausdorff dimension of the level sets of generic continuous functions defined on fractals, apart from  \cite{BBEtoph} see also \cite{BBElevel}. 
(Some people prefer to use the term typical in the Baire category sense instead 
of generic.) 
The topological Hausdorff dimension is related to some sort of "conductivity" properties of some fractal "networks" and outside of Mathematics, papers in Physics are also dealing with this concept, see for example  
\cite{Balankintoph}, \cite{Balankinfracspace}, \cite{Balankintransport}, \cite{Balankintoph2}, and  \cite{Balankinfluid}. 
 It is a
natural question to ask what happens if the level regions of our functions are not "infinitely
compressible" and hence due to thickness of the level regions we cannot use for almost
every levels the parts of our fractal domains where they are the "thinnest". The simplest
way to impose a bound on compressibility is considering Hölder functions instead of
arbitrary continuous functions. 
Motivated by this,  in this paper we consider level sets of $1$-H\"older-$\aaa$ functions
 defined on fractals.
Introducing a bound on the Hölder-constant is a customary practice as it significantly tames
the function space in question by making it complete and separable, we will discuss this a little further in Section \ref{*secprel}, where we will also 
refer to papers  \cite{ABMP} \cite{merlo} and \cite{PreissTiser} as examples where similar procedure was used. 

In Section \ref{*secprel} we give and recall some definitions 
and theorems
and prove some preliminary results.
In this section among other things we define $D_{*}(\aaa, F)$ which  is the  essential supremum of the Hausdorff dimensions of the level sets of  a generic
$1$-H\"older-$\aaa$ function.  If $F$ is the disjoint union of two fractals
$F_{1}$ and $F_{2}$, with $D_{*}(\aaa, F_1)<D_{*}(\aaa, F_2)$ then it is easy to see that it is not necessarily true that for the generic $1$-H\"older-$\aaa$  function 
$D_{*}(\aaa, F)$ equals the Hausdorff dimension of almost every level set in the range of the function. However, in Section \ref{secd*aaa} we show that for connected self-similar sets such a result holds   if $0<\alpha <1 $.  The Lipschitz case,  that is  $\alpha=1$ needs a different approach and can be the subject of some further research.

In Section \ref{*secdens} we establish some density and approximation results we need
for proving results about generic functions.

Next in Section \ref{*secub} we prove Theorem \ref{thm:trivial_upper_bound}
according to which $D_{*}(\aaa, F)$  either equals zero, or it  is always less or equal than the upper box dimension of $F$ minus one.

In Section \ref{*secex} we give the details of the  calculation of $D_{*}(\aaa,F)$ for $F$ defined in Theorem 
\ref{*thex}. 
This is an example fractal $F\sse [0,1/2]^2$,
which is a big "sponge" of positive Lebesgue measure and its complement is a dense system of very thin "tubes".  
In a "rough heuristic language" if we put our fractal sponge into $[0,1/2]^2$
then almost every level set of a typical continuous function can "run" in the complement of $F$, hence its Hausdorff dimension is $0$.
However, using Hölder level sets one can see that
$D_{*}(\aaa,F)=1$ for any $\aaa\in (0,1]$, showing that it is criss-crossed by only very "narrow" tubes and these tubes are too thin to "contain" almost
every level set of a generic $1$-H\"older-$\aaa$  function. 
For this example the calculation is relatively easy.
One can construct several examples of similar type. 
See, for example the set constructed in Theorem 6.1 
of \cite{sierc}. 
For a more complicated example, like the Sierpi\'nski triangle  in \cite{sierc} we find lower and upper estimates for $D_{*}(\aaa,F)$.
Although the Sierpi\'nski triangle $\Delta$ is connected so one cannot use
dense system of "tubes" in its complement, like in our example given in Theorem \ref{*thex},
  but near vertices of triangles used at
different levels of its construction it is not "too thick"  and almost every level set
of generic continuous functions
 defined on $\Delta$ can be squeezed into these regions, so they intersect $\Delta$ in a set of zero Hausdorff dimension.
In case of generic $1$-H\"older-$\aaa$  functions defined on $\Delta$ 
one can see that the regions where $\Delta$ "can be crossed" by zero dimensional level sets are quite limited and level sets of generic $1$-H\"older-$\aaa$ functions do not fit 
into them. Theorem 3.2 of \cite{sierc} shows that for $\aaa>0$ we have
$D_{*}(\aaa,\DDD)>0$, but contrary to our simple example given in 
Theorem \ref{*thex} of this paper by Theorem 4.5 of \cite{sierc} we also have
$D_{*}(\aaa,\DDD)<1$, for $\aaa<1$.  The estimate given there is stronger in fact, 
but this weak bound already verifies our heuristic  that 
the Sierpi\'nski triangle is blocking the level sets of a generic $1$-H\"older-$\aaa$ function, but this blockage is not as thick as the one done by the "thick sponge" of
Theorem \ref{*thex}. 

The antonym of blocking the level sets, is letting them to pass through. 
This is where "conductivity"  comes into the picture. So parts of the fractal
where level sets can "go through" with an intersection of small Hausdorff dimension are "well conducting".
Trying to find these areas to get the lower estimate of $D_{*}(\aaa,\DDD)$
to prove Theorem 3.2 of \cite{sierc} we define a concept of conductivity of subtriangles  used at different levels of the definition of the Sierpi\'nski
triangle.

In Section \ref{*secdgd}  we prove Theorem \ref{thm:generic_existence},
that is we show that  
there is a dense $G_\delta$ subset $\cag$ of the 
$1$-H\"older-$\aaa$ functions such that for every $f\in\cag$   the
essential supremum of the Hausdorff dimensions of the level sets of 
$f$
equals $D_*(\alpha,F)$.
This shows that in the complicated looking definition of $D_*(\alpha,F)$, in \eqref{*defDcsaF} the supremum is a maximum and the infimum is a minimum which equals the maximum, that is
$D_{*}^{f}$ is  not depending on $f\in \cag$.

In Section \ref{*secmon} we verify
Theorem \ref{thm:generic_monotonicity}, that is we show that $D_{*}(\aaa, F)$ is monotone increasing
in $\aaa$ for any compact set $F$. 

 Finally, in Section \ref{secd*aaa} we 
prove Theorem \ref{*thessup}, stating
 that if $F$ is a  connected self similar set, and $0<\alpha<1$, then 
one can select a dense $G_{\ddd}$ set such that for any $f\in \cag$
for almost every $r\in f(F)$
the Hausdorff dimension of the level set $f^{-1}(r)$ equals
$D_*(\aaa,F)$.  

This paper contains our first results about generic H\"older level sets on fractals.
We have one more paper \cite{sierc} which is almost finished.  While in the first paper we
primarily focus on establishing the fundamental properties of the examined quantities and hence lay the 
foundations of this theory, the other paper is dedicated to actual calculations and estimations. 
 In that paper, as an example, we  present quantitative results about 
$D_{*}(\aaa, F)$ when $F$ is the Sierpinski triangle.
We also show that for  certain  "strongly" separated fractals  $D_{*}(\aaa, F)$ 
is zero for small values of $\aaa$.
 In particular,  $D_{*}(\aaa, F)=0$ for all self-similar fractals satisfying the strong
separation condition and  $0<\alpha<1$ .
 We note that strongly separated fractal sets are of zero topological dimension and it is
well-known that on such sets the generic continuous functions are one-to-one, so every non-empty level set is a
singleton, and hence it is of zero dimension.
We also  give an example of a  "strongly" separated fractal,
  $F$ such that 
$D_{*}(\aaa, F)=0$  only  for $\aaa<1/2$.  
This means at a heuristic level that for such fractals the level sets of 
generic 1-H\"older-$\aaa$ functions are as flexible/compressible as those
of a continuous function. 
 On the other hand,  for $\aaa>1/2$ we have $D_{*}(\aaa, F)=1$, that is after a critical value of $\aaa$ these level sets 
are not as flexible/compressible as those of a continuous function.
So for small values of $\aaa$ only the geometry of $F$ matters, while for
larger $\aaa$s the H\"older exponent also counts.
We call this phenomenon phase transition.
 As we have mentioned,  $D_{*}(\aaa, F)$ is monotone increasing
in $\aaa$ for any compact set $F$. 
It is a natural question whether this function is continuous.
For the fractal $F$ mentioned above there is not only a 
phase transition at $\aaa_{\phi}=\frac{1}{2}$, but also a jump discontinuity
of $D_{*}(\aaa, F)$. This set is a less thick "fractal sponge" than the one discussed in Theorem \ref{*thex} of this paper.

%From now on \emph{generic} is always understood in the sense of Baire category.

The study of generic continuous functions has a long history.
Topological description of the generic continuous function
defined on $[0,1]$ was given by Bruckner and Garg \cite{BG}.
 Dimensions of the graphs of generic continuous functions were 
 studied by Mauldin and Williams \cite{MW}, Humke and Petruska \cite{HP},
Hyde, Laschos, Olsen, Petrykiewicz, and Shaw \cite{HLOPS}. 
Kirchheim \cite{BK} showed $\dim_{H}f^{-1}(y)=d-1$ for 
the generic continuous function $f\colon [0,1]^d\to \mathbb{R}$
for all $y$ in the interior of the range of $f$.

In \cite{BBEtoph} the topological Hausdorff dimension,
$\dim_{tH} K$ of a compact metric space $K$ was introduced.
It was proved that topological Hausdorff dimension  describes the Hausdorff dimension of
the level sets of the generic continuous function on $K$, namely 
  $\sup\{
\dim_{H}f^{-1}(y) : y \in \R \} = \dim_{tH} K - 1$ for the generic
$f \in C(K)$. 
If $K$  is
self-similar then one can say more, for the generic $f\in C(K)$ we have $\dim_{H} f^{-1}(y)=\dim_{tH}K-1$
 for every $y$
in the interior of $f(K)$.

In \cite{MaZha}  topological Hausdorff dimension of fractal squares was considered.

%%%%%%%%%%%%%%%%%%%%%%%%%%%%%%%%
%%%%%%%%%%%%%%%%%%%%%%%%%%%%%%%%
%%%%%%%%%%%%%%%%%%%%%%%%%%%%%%%%
%%%%%%%%%%%%%%%%%%%%%%%%%%%%%%%%
%%%%%%%%%%%%%%%%%%%%%%%%%%%%%%%%
%%%%%%%%%%%%%%%%%%%%%%%%%%%%%%%%

\section{Preliminaries}\label{*secprel}

The distance of $x,y\in \R^{p}$ is denoted by $|x-y|$. 
If $A\sse \R^{p}$ then the diameter of $A$ is denoted by $|A|=\sup\{ |x-y|: x,y\in A \}.$
The open ball of radius $\varrho$ centered at $x$ is denoted by $B(x,\varrho)$.
For a set $E\sse \R^{p}$ its $\varrho$-neighborhood $\{ x: \inf\{ |x-y|: y\in E \}<\varrho \}$ is denoted  by $U_\varrho(E)$. 

Assume that $F\subseteq \mathbb{R}^p$ for some $p>0$. 
In what follows, $F$ will be some fractal set, usually we suppose that it is compact. 

We say that a function $f: F \to \mathbb{R}$ is $c$-H\"older-$\alpha$ for $c>0$ and $0<\alpha\leq 1$ if $|f(x)-f(y)|\leq c|x-y|^\alpha$. 
The space of such functions will be denoted by $C^{\aaa}_{c}(F)$, or if $F$ is fixed then by $C^{\aaa}_{c}$. 
The space of H\"older-$\aaa$ functions will be denoted by $C^{\aaa}$, that is $C^{\aaa}=\bigcup_{c}C^{\aaa}_{c}.$
We say that $f$ is $c^-$-H\"older-$\alpha$ if there exists $c'<c$ such that
$f$ is $c'$-H\"older-$\alpha$. The set of such functions is denoted by $C^{\aaa}_{c^{-}}$, that is $C^{\aaa}_{c^{-}}=\bigcup_{c'<c}C^{\aaa}_{c'}$.

In the space of H\"older-$\aaa$ functions often  the norm
$$ ||f||_{C^{0,\aaa}}=||f||_{\oo}+\sup_{x,y\in F,\ x\not=y}\frac{|f(x)-f(y)|}{|x-y|^{\aaa}}$$
is considered. 
This is a Banach space and one can consider typical properties in these spaces as well. 
However, these spaces are usually non-separable and often it is more convenient to consider  
H\"older functions as subsets of continuous functions equipped with the supremum norm $ ||f||_{\oo}=\sup_{x\in F}|f(x)|.$ 
To obtain a closed subset of $C^{\aaa}(F)$ we will consider $1$-H\"older-$\alpha$ functions, $C^{\aaa}_{1}(F)$ and use the metric coming from the supremum norm.
One could use $C^{\aaa}_{c}(F)$  with any fixed positive constant $c$ instead of $1$. The results would be the same. 
In Lipschitz and H\"older spaces it is not unusual to consider these spaces.
For example in \cite{PreissTiser} and \cite{merlo} the one Lipschitz cases, in our notation $C^{1}_{1}([0,1])$ and $C^{1}_{1}([0,1]^n)$  were used.
In Theorem 2.13 of \cite{ABMP} generic results in the spaces $C^{\aaa}_{1}([0,1])$, $0<\aaa<1$
were considered, even our notation is identical to the one used there.

For $\rrr>0$ and $f\in C(F)$ we denote by $B(f,\rrr)$ the open ball of radius $\rrr$
centered at $f$, the ball taken in the supremum norm. 
If $f\in C_{1}^{\aaa}(F)$ then
$B(f,\rrr)\cap C_{1}^{\aaa}(F)$ will denote the corresponding open ball in the subspace $C_{1}^{\aaa}(F).$

Since similarities are not changing the geometry of a fractal set to avoid some 
unnecessary technical difficulties we suppose that we work with fractal sets $F$
of diameter  not exceeding  one. 
This way
\begin{equation}\label{*caaaineq}
C^{\aaa}_{1}(F)\sse C^{\aaa'}_{1}(F) \text{  if $\aaa>\aaa'$.}
\end{equation}

Suppose $A\sse \R^{p}$. 
Given $\ddd>0$  we say that the sets $U_{j}$ form a $\ddd\text{-cover}$ of $A$ 
if $|U_{j}|<\ddd$ for all $j$ and $A\sse \bigcup_{j} U_{j}$.
 
The $s$-dimensional Hausdorff measure (see its definition for example in \cite{[Fa1]}) is denoted by $\cah^{s}$. 
Recall that the Hausdorff dimension of $A\sse \R^{p}$
is given by
 \begin{equation}\label{*defdimh}
  \dim_H A=\inf\{ s : \cah^{s}(A)=0\}=
\end{equation}
  $$\inf\{s:\ex \mathbf{C}_s>0,\ \ax\ddd>0,\  \ex \{ U_{j} \} \text{ a } \ddd\text{-cover of }A  \text{ s.t. } \sum_{j}|U_{j}|^{s}<\mathbf{C}_s \}.$$
  
 One can observe that in the above definition instead of arbitrary  $\ddd\text{-covers of } A$ one can use open $\ddd\text{-covers}$, that is we can assume that the sets $U_{j}$ are open. 

 \begin{comment}
xxxxxxxxxxxxxxxxxxxxxxxxxxxxxxxxxxxxxxxxx  
  
We will need an equivalent definition of the Hausdorff dimension. 
Denote by 
$\ds \frq_{n,0,d}$ the dsystem of open dyadic grid cubes at level $n$, that is 
$$ \frq_{n,0,d}=\{ \prod_{j=1}^{d} ((k_{j}-1)2^{-n},k_{j}2^{-n}): k_{j}\in\Z,\ j=1,...,d\}.$$
Since these open cubes do not cover $\R^{p}$ we also take a translated system
$$ \frq_{n,0,d}=\{ Q+(2^{-n-1},...,2^{-n-1}): Q\in \frq_{n,0,d}\}.$$
Finally the union of these systems $\frq_{n,d}=\frq_{n,0,d}\cup \frq_{n,1,d}$ is an open cover of $\R^{p}$. 
In case the dimension $d$ is fixed we omit $d$ from the notation.

Finally we set $\frq=\frq_{d}=\bigcup_{n}\frq_{n}$.

Suppose $A\sse \R^{p}$. 
Given $\ddd>0$ the sets $U_{j}$, $\ddd-\frq-\text{cover }A$
if $|U_{j}|<\ddd$, $U_{j}\in \frq$ for all $j$ and $A\sse \bigcup_{j} U_{j}$.

It is an easy exercise left to the reader that we can use these covers in the definition of Hausdorff dimension, that is
\begin{equation}\label{*defdimhfrq}
  \dim_H A=
\inf\{s:\ex \mathbf{C}_s>0,\ \ax\ddd>0,\  \ex \{ U_{j} \} \text{ a } \ddd-\frq-\text{cover of }A  \text{ s.t. } \sum_{j}|U_{j}|^{s}<\mathbf{C}_s \}.
\end{equation}

xxxxxxxxxxxxxxxxxxxxxxxxxxxxxxxxxxxxxxx
\end{comment}

Since the topological Hausdorff dimension is a less known concept
here we quickly mention some definitions and results from \cite{BBEtoph}. 
First we recall the definition of the
(small inductive) topological dimension.

\begin{definition} Set $\dim _{t} \emptyset = -1$. 
The \emph{topological dimension}
of a non-empty metric space $X$ is defined by induction as
$$
\dim_{t} X=\inf\{d: X \textrm{ has a basis } \iU \textrm{ such that
} \dim_{t} \partial {U} \leq d-1 \textrm{ for every } U\in \iU \}.
$$
\end{definition}

The topological Hausdorff dimension is  defined analogously to the topological
dimension. 

 In the next definition we adopt the convention that
$\dim_{H}\emptyset = -1$.

\begin{definition}\label{deftoph}
Set $\dim _{tH} \emptyset=-1$. 
The \emph{topological Hausdorff
dimension} of a non-empty metric space $X$ is defined as
\[
\dim_{tH} X=\inf\{d:
 X \textrm{ has a basis } \iU \textrm{ such that } \dim_{H} \partial {U} \leq d-1 \textrm{ for every }
 U\in \iU \}.
\]
\end{definition}

Both notions of dimension can attain the value $\infty$ as well.

If $K$ is a compact metric space and $\dim_{t}K=0$ then the generic $f\in C(K)$ is
well-known to be one-to-one, so every non-empty level set is a
singleton. We do not know where this folklore fact was first proved but its simple proof can be found for example in 
\cite{BBElevel}.

Assume $\dim_{t}K>0$. 
The following results from \cite{BBEtoph} show the
connection between the topological Hausdorff dimension and the level
sets of the generic $f\in C(K)$.

%%%%%%%%%%%%%%%%%%%%%%%%%%%%%%%%
%%%%%%%%%%%%%%%%%%%%%%%%%%%%%%%%
%%%%%%%%%%%%%%%%%%%%%%%%%%%%%%%%
%%%%%%%%%%%%%%%%%%%%%%%%%%%%%%%%

\begin{theorem} \label{ft}
If $K$ is a compact metric space with $\dim_{t}K>0$ then for the
generic $f\in C(K)$
\begin{enumerate}[(i)]
\item $\dim_{H} f^{-1} (y)\leq \dim_{tH} K-1$ for every $y\in \R$,
\item for every $\varepsilon > 0$ there exists an interval $I_{f,\varepsilon}$
such that $\dim_{H} f^{-1} (y)\geq
\dim_{tH} K - 1 - \varepsilon$ for every $y\in I_{f,\varepsilon}$.
\end{enumerate}
\end{theorem}

\begin{corollary} \label{sup}
If $K$ is a compact metric space with $\dim_t K > 0$ then
$\sup\{ \dim_{H}f^{-1}(y) : y \in \R \} = \dim_{tH} K - 1$ for the
generic $f \in C(K)$.
\end{corollary}

There are many  equivalent  definitions of the box or Minkowski dimension.
We will use the following one:

\begin{definition}\label{*defboxd}
Given a non-empty set $F\sse \R^p$
let $a_N(F)$ denote the number of  closed  $2^{-N}$ grid hypercubes 
intersected by $F$. 
The lower and upper box dimensions of $F$
equal $\ds \ldimb F=\liminf_{N\to\oo}\frac{\log a_{N}(F)}{N\log 2}$,
$\ds \udimb F=\limsup_{N\to\oo}\frac{\log a_{N}(F)}{N\log 2}$.
If $\ldimb F=\udimb F$ then this common value is the box dimension of $F$, denoted by $\dimb F$.  
For an empty set $F$ we put $\ldimb F=\udimb F=\dimb F=0$. 
\end{definition}

The above definition makes sense for an arbitrary set of $F\sse \R^p$,
but in this paper we will mainly work with measurable sets.
 
We need approximations by smooth functions.
We will use the bump function\begin{equation} \eta(x) = 
	\begin{cases*}
      \exp\left(-\frac{1}{1-|x|^2}\right) & if $|x|<1$, \\
      0        & otherwise,
    \end{cases*}
\end{equation}
and the corresponding mollifier 
\begin{displaymath}
\eta_{r}(x)=c_{r}\eta\left(\frac{x}{r}\right),
\end{displaymath}
where $c_r$ is defined such that $\int_{\mathbb{R}^p}\eta_r(x)dx=1$.

We want to study the Hausdorff dimension of the level sets of arbitrary $1$-H\"older-$\alpha$  functions  and 
also of the generic $1$-H\"older-$\alpha$ functions.
 
To make it more precise, we introduce the following notation: let $D^f(r,F)=D^f(r)=\dim_H(f^{-1}(r))$ for any function $f: F\to \mathbb{R}$, that is $D^f(r)$ denotes the Hausdorff dimension of the function $f$ at level $r$.

We are interested in those values
for which the level-set is of large Hausdorff dimension for many level sets in the sense of Lebesgue measure. 
This motivates the following definition. 
\begin{displaymath}
D_{*}^f(F)=D_{*}^f=\sup\{d: \lambda\{r : D^f(r,F)\geq{d}\}>0\},
\end{displaymath}
where $\lambda$ denotes the one-dimensional Lebesgue measure.
Later we will assume that our fractal $F$ is compact, but the above definition makes sense for more general measurable sets as well.

% A function is locally non-constant on $F$ if for any $\eee>0$ and for any $x$ which is not an isolated point of $F$
%there is $y\in F\cap B(x,\eee)$  such that $f(x)\not=f(y)$. 

 The definition of $D_*^f(F)$  depends on $f$. 
In case we want a definition depending only on the fractal $F$ we can first take 
\begin{displaymath}
\underline{D}_{*}(\alpha,F)=\inf\{D_{*}^f: f\colon F\to\mathbb{R} \text{ is  locally  non-constant and $1$-H\"older-}\alpha\},
\end{displaymath}
where the locally non-constant property is understood as $f$ is non-constant on $U\cap F$ where $U$ is  any  neighborhood of  any  accumulation point of $F$. %(Notably, if $F$ is connected, the condition is equivalent to $f$ being non-constant.)
As we are only concerned with nonnegative numbers, by convention the infimum of the empty set is $0$. 
The value $\underline{D}_{*}(\alpha,F)$ concerns those functions for which 
"most" level sets are smallest possible.

As mentioned earlier we are also interested in level sets of generic 
$1$-H\"older-$\alpha$ functions. 

We denote by $\mg_{1,\aaa}(F)$, or by simply $\mg_{1,\aaa}$ the system of dense $G_{\ddd}$
sets in $C_{1}^{\aaa}(F)$.

We put
\begin{equation}\label{*defDcsaF}
D_{*}(\aaa,F)=\sup_{\cag\in \mg_{1,\aaa}}\inf\{ D_{*}^{f}:f\in \cag \}.
\end{equation}
In Theorem \ref{thm:generic_existence} we will show that there is a $G_\delta$ subset $\cag$ of $C_1^\alpha(F)$ such that for every $f\in\cag$ we have $D_*^f(F) = D_*(\alpha,F)$.

As we remarked in the introduction the existence of the above $\cag$ shows that
in the above definition the supremum is maximum, taken at this $\cag\in \mg_{1,\aaa}$,
and for this special $\cag$ there is no need to take the infimum, since 
  $D_{*}^{f}$ takes this minimum for any $f\in\cag$, which at the same time equals the maximum value.
   Combined with Theorem \ref{*thessup}  for $0<\alpha<1$  in case of connected self-similar fractals  one can think of $D_{*}(\aaa,F)$ as
the Hausdorff dimension of almost every level set in the range of the generic 
$C_1^\alpha(F)$
function.

So far we have considered $0<\aaa\leq 1$. 
To include generic continuous functions 
in our notation we set  $D_{*}(0, F)=\max\{0,\dim_{tH} F-1\}$. 
By Theorem \ref{ft},
 if $f$ is the generic continuous function on $F$, then 
$$D_{*}(0, F)=D_{*}^f (F) .$$

 For brevity, often we will omit $F$ from our notation.

%%%%%%%%%%%%%%%%%%%%%%%%%%%%%%%%
%%%%%%%%%%%%%%%%%%%%%%%%%%%%%%%%
%%%%%%%%%%%%%%%%%%%%%%%%%%%%%%%%
%%%%%%%%%%%%%%%%%%%%%%%%%%%%%%%%
%%%%%%%%%%%%%%%%%%%%%%%%%%%%%%%%
%%%%%%%%%%%%%%%%%%%%%%%%%%%%%%%%

\section{Main Results}\label{*secmain}

First we give a trivial upper bound for $D_*(\alpha, F)$.
Observe that this upper bound does not depend on $\aaa$.

\begin{theorem} \label{thm:trivial_upper_bound}
For any  bounded  measurable set $F\subseteq \mathbb{R}^{p}$, we have 
$$
D_*(\alpha, F) \leq \max\{0,\overline{\dim}_B (F) - 1\}.
$$ 
\end{theorem}

Next we give an example for computation of $D_*(\alpha, F)$.

\begin{theorem}\label{*thex}
Set $G_k:=\bigcup_{j\in\Z} \left(j\cdot2^{-k^2},j\cdot 2^{-k^2}+2^{-k^3}\right)$ for every $k\in\N$, 
$$
F_0 := [0,1/2]\setminus\bigcup_{k=2}^\infty G_k
$$ 
and $F:=F_0\times F_0$.
%Set 
%$$
%F_0:=[0,1/2]\setminus\left(\bigcup_{k=2}^\infty\bigcup_{j\in\Z} \left(2j\cdot2^{-k^2},2j\cdot 2^{-k^2}+2^{-2k^2}\right)\right)
%$$
%and $F:=F_0\times F_0$. 
For every $\alpha\in(0,1]$ we have $D_*(\alpha,F)=1$, and $D_*(0,F)=0$.
\end{theorem}

The closed set $F$ defined above almost "fills out" $[0,1/2]^2$. We  have selected $[0,1/2]^2$, since we wanted to have a set of diameter not exceeding $1$. 
It is looking like a "sponge" there is a dense system of narrow tubes in it and it is of zero topological dimension. If one considers the function $f_0(x,y)=y$ then its level sets are horizontal, running West-East. Taking a ``generic continuous function" $f\in C^{0}(F)$ close to
$f_{0}|_{F}$ almost all of its level sets  are empty. We can also interpret it in
the following way. Take a continuous extension of $f$ onto $[0,1/2]^2$, still denoted by $f$. Then  its level sets are still "running almost West-East" 
but they are "flexible and compressible enough" to stay in the complement of $F$.
This means that the topological Hausdorff dimension is not "sensing" the fact that
$F$ is a "large sponge". On the other hand,  the theorem tells us that the  level sets of generic H\"older-$\aaa$ functions cannot be squeezed into the thin tubes in the complement of
$F$, this is reflected by the fact that $D_*(\alpha,F)=1$ when $0<\aaa\leq 1.$
For this fractal it is easy to carry out the calculations. 

In case of connected fractals like the Sierpi\'nski triangle or the Sierpi\'nski carpet
there are no tubes/holes in the complement in the fractal, but there are parts
where it is thinner and there are parts where it is thicker.
Level sets of generic $1$-H\"older-$\aaa$ functions "try to run" at parts of the fractal
where "it is thin". They give more precise information about these properties of the fractal than topological Hausdorff dimension.
In paper \cite{sierc} we give estimates on $D_*(\alpha,F)$ for the Sierpi\'nski triangle.

Next theorem shows that in the complicated looking definition \eqref{*defDcsaF} for a suitable $G_{\delta}$ set one can skip taking $\inf$ and $\sup$.

\begin{theorem} \label{thm:generic_existence}
If   $0< \aaa\leq 1$  and $F\subset\R^p$ is compact, then there is a dense $G_\delta$ subset $\cag$ of $C_1^\alpha(F)$ such that for every $f\in\cag$ we have $D_*^f(F) = D_*(\alpha,F)$.
\end{theorem}

From \eqref{*caaaineq} it follows that $\underline{D}_{*}(\alpha,F)$ is monotone increasing in $\aaa$, that is $\underline{D}_{*}(\alpha,F)\leq \underline{D}_{*}(\alpha',F)$
if $\aaa\leq \aaa'$. 
Next we state  the same property for  $D_{*}(\alpha, F)$.

\begin{theorem} \label{thm:generic_monotonicity}
Suppose that $F\subset\R^p$ is compact. Then the function $D_*(\alpha,F)$ is monotone increasing in $\alpha$  on $ (0,1]$.
\end{theorem}

Assume that $F$ is a self-similar set  determined by the contractive similarities $\fff_1,...,\fff_m$, $m\geq 2$ with ratios $0<\mathbf{q}_1, ..., \mathbf{q}_m<1$,
that is, $\ds F=\bigcup_{i}\fff_{i}(F)$. 
We put $\mathbf{q}_{\min}=\min\{ \mathbf{q}_1, ..., \mathbf{q}_m \}$.
We do not assume the Open Set Condition.

\begin{theorem}\label{*thessup}
Suppose that $F$ is a connected self-similar set and $0<\alpha<1$. 
Then there exists a dense $G_{\ddd}$ set $\cag$ in $C_{1}^{\aaa}(F)$ such that for any $f\in \cag$ 
$$D_*(\aaa,F)=D_{*}^f(F)=D^f(r,F)\text{ for a.e. }r\in f(F).$$
\end{theorem}

This shows that in case of connected self-similar sets, like the Sierpi\'nski triangle or the Sierpi\'nski carpet one can think of $D_*(\aaa,F)$
as the Hausdorff dimension of almost every level set in the range of a generic continuous function.

%%%%%%%%%%%%%%%%%%%%%%%%%%%%%%%%
%%%%%%%%%%%%%%%%%%%%%%%%%%%%%%%%
%%%%%%%%%%%%%%%%%%%%%%%%%%%%%%%%
%%%%%%%%%%%%%%%%%%%%%%%%%%%%%%%%
%%%%%%%%%%%%%%%%%%%%%%%%%%%%%%%%
%%%%%%%%%%%%%%%%%%%%%%%%%%%%%%%%

\section{Some approximation and density results }\label{*secdens}

We recall an extension theorem which is a consequence of  Theorem 1 of \cite{[GrunbHolderext]}.

\begin{theorem}\label{*Grunb}
Suppose that $F\sse \R^{p}$ and $f:F\to \R$ is a $c$-H\"older-$\alpha$ function.
Then there exists a $c$-H\"older-$\alpha$ function $g:\R^{p}\to \R$
 such that
 $g(x)=f(x)$ for $x\in F$.
\end{theorem}

Next we prove the following general lemma, which will turn out to be rather useful in the study of generic properties of H\"older functions:

%%%%%%%%%%%%%%%%%%%%%%%%%%%%%%%%
%%%%%%%%%%%%%%%%%%%%%%%%%%%%%%%%
%%%%%%%%%%%%%%%%%%%%%%%%%%%%%%%%

\begin{lemma} \label{lipschitzapprox}
Assume that $F$ is compact and $c>0$ is fixed. 
Then the Lipschitz $c$-H\"older-$\alpha$ functions defined on $F$ form a dense subset of the $c$-H\"older-$\alpha$ functions.
\end{lemma}

\begin{proof}
Consider an arbitrary $c$-H\"older-$\alpha$ function $f:F\to\mathbb{R}$ and fix  $\varepsilon>0$. 
 By using Theorem \ref{*Grunb} we
 extend $f$ to $\mathbb{R}^p$.  The  $c$-H\"older-$\alpha$ function obtained this way will be still denoted by $f$. 
%Denote the extension by $f$ as well. 
It is known by the theory of mollifiers that if we consider the convolution $f_r = f * \eta_r$, it is a $C^{\infty}$ function and $f_r \to f$ in the supremum norm on any compact subset of $\R^{p}$ as $r\to 0+$. 
Moreover, $f_r$ restricted to $F$ is $c$-H\"older-$\alpha$ as well. 
Indeed, for $x,y\in F$, due to the triangle inequality and the fact that the support of $\eta_r$ is $\{|z|<r\}$,
\begin{equation} \begin{split} |f_r(x)-f_r(y)| &=\left|\int_{\mathbb{R}^p} \eta_r(z)f(x-z)dz - \int_{\mathbb{R}^p} \eta_r(z)f(y-z)dz\right| \\ &\leq \int_{\mathbb{R}^p}\eta_r(z)|f(x-z)-f(y-z)|dz \\ &= \int_{\{|z|<r\}}\eta_r(z)|f(x-z)-f(y-z)|dz \\ &\leq c |x-y|^{\alpha}\int_{\{|z|<r\}}\eta_r(z)dz \\ &= c |x-y|^{\alpha}. \end{split} \end{equation}
%where the last inequality holds for small enough $r$ guaranteeing $x-z, y-z \in F'$. 
Consequently, we can fix $r$ such that the restriction  of $f_r$ to $F$ is a $c$-H\"older-$\alpha$ function in the $\varepsilon$-neighborhood of the restriction $f$ to $F$ in the supremum norm. 
Suppose that $F'$ is  a  compact convex set containing $F$.
As $f_r$ is smooth, its derivative on $F'$ is bounded. 
Consequently, $f_r$ is $K$-Lipschitz on $F'\supset F$ for some $K>0$.
\end{proof}

Approximations by piecewise affine functions in the space $C_1^{\alpha}(F)$ are important as well. We will prove a lemma of this nature, but in order to avoid ambiguity, we first provide a precise definition:

%%%%%%%%%%%%%%%%%%%%%%%%%%%%%%%%
%%%%%%%%%%%%%%%%%%%%%%%%%%%%%%%%
%%%%%%%%%%%%%%%%%%%%%%%%%%%%%%%%

\begin{definition}\label{*defpcaff}
A function $f:F\to\R$
is piecewise affine on $F\sse \R^{p}$,
if  we can find a system $\frs$ of non-overlapping (means disjoint interiors), non-degenerate closed $p$-simplices 
such that $F\sse \bigcup_{S\in \frs}S$, 
the set $\{S\in\frs :  S\cap B \neq\emptyset\}$ is finite  for every bounded $B\subset\R^p$, 
 and 
for any $S\in \frs$
the restriction of $f$ to any $S\cap F$
coincides with the restriction of an affine function to $S\cap F$.

%We say that the piecewise  affine function $f$ is locally non-constant if the affine functions appearing in the above definition can be chosen to be all non-constant.  
\end{definition}

\begin{lemma} \label{piecewise_affine_approx}
Assume that $F$ is compact, $0<\alpha<1$, and $0<c$ are fixed. 
Then the locally non-constant piecewise affine $c^{-}$-H\"older-$\alpha$ functions defined on $F$ form a dense subset of the $c$-H\"older-$\alpha$ functions.
\end{lemma}

Before proving this lemma, we state and prove an auxiliary proposition which is surely known in some form:

%%%%%%%%%%%%%%%%%%%%%%%%%%%%%%%%
%%%%%%%%%%%%%%%%%%%%%%%%%%%%%%%%
%%%%%%%%%%%%%%%%%%%%%%%%%%%%%%%%

\begin{proposition} \label{affine_approx_simplex}
Assume that $S\subseteq \mathbb{R}^p$ is a non-degenerate $p$-simplex with vertices $x_0, ..., x_p$, and $\widetilde{f}:\{x_0, ..., x_p \} \to \mathbb{R}$ is $K$-Lipschitz for some $K>0$. 
Let $a>0$ be the length of the longest edge of $S$, and 
let $b=\min_{0\leq i \leq  p} b_i$, where $b_i>0$ is the distance between $x_i$ and the hyperplane determined by the remaining vertices. 
Then the function $\overline{f}: S\to \mathbb{R}$ defined by
$$ \overline{f}(x) = \sum_{i=0}^{p} \gamma_i \widetilde{f}(x_i)$$
for any convex combination $x=\sum_{i=0}^{p} \gamma_i x_i$ is $M$-Lipschitz, where
$$M= (p+1)\cdot K  \cdot \frac{a}{b},$$
that is $M$ depends on $S$ only through $\frac{a}{b}$. 
In particular, it is invariant  with respect  to similarities.
\end{proposition}

\begin{proof}
As $S$ is the convex hull of its vertices and any two vertices are connected by an edge, its diameter equals $a$. 
Moreover, adding a constant to $\widetilde{f}$ does not change the assumption, nor the implication of the proposition. 
Consequently, we can assume that $\min{\widetilde{f}} = 0$, and hence by the $K$-Lipschitz property we have $\max{\widetilde{f}}\leq Ka$.

Consider now arbitrary points $x,x'$ in $S$ with
$$x=\sum_{i=0}^{p} \gamma_i x_i,\quad x'=\sum_{i=0}^{p} \gamma_i' x_i.$$
Without loss of generality, we can assume that $|\gamma_0 - \gamma_0'|$ is the maximal amongst the differences $|\gamma_i - \gamma_i'|$, as $i= 0, 1, ..., p$.
Then
\begin{equation} \label{tilde_f_difference}
|\overline{f}(x) - \overline{f}(x')| \leq \sum_{i=0}^{p} |\gamma_i - \gamma_i'| \widetilde{f}(x_i) \leq (p+1)\cdot|\gamma_0 - \gamma_0'|\cdot Ka
\end{equation}
where we use the bound on $\widetilde{f}$ in the last inequality. 
This quantity should be compared to the distance $|x-x'|$ to check the Lipschitz property of $\overline{f}$. 
However, one can easily see 
that the distance of $x$ from the hyperplane determined by $x_1, ..., x_p$ is $\gamma_0 b_0$, while the distance of $x'$ from the same hyperplane is $\gamma_0' b_0$. 
Consequently,
\begin{equation} \label{distance_estimate}
|x-x'|\geq |\gamma_0 - \gamma_0'| b_0 \geq |\gamma_0 - \gamma_0'| b.
\end{equation}
Comparing estimates \eqref{tilde_f_difference} and \eqref{distance_estimate}, we obtain
\begin{displaymath}
|\overline{f}(x) - \overline{f}(x')| \leq (p+1)\cdot K \cdot \frac{a}{b}|x-x'| = M|x-x'|. 
\end{displaymath}
\end{proof}

\begin{proof}[Proof of Lemma \ref{piecewise_affine_approx}]

Consider an arbitrary $c$-H\"older-$\alpha$ function $f:F\to\mathbb{R}$ and fix  $\varepsilon>0$. Since $F$ is compact we can choose $0<\gggg<1$
such that $\left\|f- \gggg f\right \|_\infty <\varepsilon/4$. Then $\gggg f$
is $c'$-H\"older-$\aaa$ on $F$ with $c'=c\gggg<c$.
The proof starts similarly to the proof of Lemma \ref{lipschitzapprox}:  using Theorem \ref{*Grunb} we extend $\ggg f$ to $\mathbb{R}^p$ such that it is still $c'$-H\"older-$\alpha$. We select  a closed hypercube $F'$ containg $F$
in its interior.
By Lemma \ref{lipschitzapprox}  we can find a $K$-Lipschitz, $c'$-H\"older-$\alpha$ function $\widetilde{f}$ for some $K>0$ with domain $F'$ such that on $F'$ we have
$$\left\|\widetilde{f}- \gggg f\right \|_\infty <\frac{\varepsilon}{4},%$$
\text{ which implies }
%$$
\left\|\widetilde{f}- f\right \|_\infty <\frac{\varepsilon}{2}.$$
 By introducing a further perturbation to $\widetilde{f}$ 
we will obtain a piecewise affine $c$-H\"older-$\alpha$ function $\overline{f}$ satisfying
\begin{equation}\label{*epscl}
\left\|\overline{f}- f\right \|_\infty <\varepsilon \text{ on }F.
\end{equation}
 To this end, fix any finite subdivision $\mathcal{U}$ of the unit hypercube into non-overlapping, non-degenerate $p$-simplices.
(The existence of such a  simplicial subdivision  is simple to see.)
Now divide $F'$ into uniform, non-overlapping hypercubes such that their diameter is below some constant $\delta>0$ to be fixed later. 
Let us divide these hypercubes further according to $\mathcal{U}$, that is denoting by $\Phi_Q$ a similarity from the unit hypercube onto a hypercube $Q$ take the subdivision $\{ \Phi_Q(S): S\in  \mathcal{U}\}$. 
Now if a simplex arising from this decomposition of $F'$ has vertices $x_0, ..., x_p$, for any convex combination $x=\sum_{i=0}^{p} \gamma_i x_i$ let
$$ \overline{f}(x) = \sum_{i=0}^{p} \gamma_i \widetilde{f}(x_i).$$
 Observe that 
\begin{equation}\label{*offdist}
|\overline{f}(x)-f(x)|\leq  \sum_{i=0}^{p} \gamma_i |\widetilde{f}(x_i)-\widetilde{f}(x)|+|\widetilde{f}(x)-f(x)|\leq K\delta+\frac{\eee}{2}<\frac{3\eee}{4},
\end{equation}
if $\delta< {\eee}/(4K) $. 
According to Proposition \ref{affine_approx_simplex}, the resulting function $\overline{f}$ is Lipschitz restricted to any of the small simplices, where the Lipschitz constant is invariant to similarities. 
However any of these small simplices is similar to a simplex 
$S\in \mathcal{U}$,  and  as $\mathcal{U}$ is finite, there are finitely many such $S$s. 
Consequently, we can choose some $M$ independently from $\delta$, such that $\overline{f}$ is $M$-Lipschitz restricted to any of the small simplices. 
Hence $\overline{f}$ is clearly $M$-Lipschitz on $F'$ as well, since any line segment in $F'$ is the finite union of line segments contained by small simplices. 

Choose and fix $c''\in (c',c)$.
Consider now arbitrary $x,y \in F'$. 
Due to the Lipschitz property of $\overline{f}$,
$$|\overline{f}(x) - \overline{f}(y)| \leq M|x-y| =M|x-y|^{1-\aaa}|x-y|^{\aaa} \leq c'' |x-y|^{\alpha}$$
if $|x-y| \leq \left(\frac{c''}{M}\right)^{\frac{1}{1-\alpha}}$. 
That is, if $x,y$ are close enough, the desired H\"older bound holds. 
Hence in what follows we can restrict our arguments to $x,y$ with $|x-y|> \left(\frac{c''}{M}\right)^{\frac{1}{1-\alpha}},$ bounded away from $0$.

 We can find  vertices $x', y'$  of the small simplices which are at most $\delta$ apart from $x, y$, respectively.  We have that 
$$|\overline{f}(x) - \overline{f}(y)| \leq |\overline{f}(x) - \overline{f}(x')| + |\overline{f}(x') - \overline{f}(y')| + |\overline{f}(y') - \overline{f}(y)|.$$
By estimating the first and the third term using the Lipschitz bound, and the second term using the H\"older bound (as $\overline{f}(x')=\widetilde{f}(x')$ and $\overline{f}(y')=\widetilde{f}(y')$),  we obtain
$$|\overline{f}(x) - \overline{f}(y)|\leq 2M \delta + c'|x' - y'|^{\alpha} \leq 2M \delta + c'(|x - y| + 2\delta )^{\alpha}.$$
As $\delta\to 0+$, the expression on the right hand side tends to $c'|x-y|^{\alpha}$. 
Consequently, as $|x-y|$ is bounded away from $0$, for small enough $\delta$ it is always smaller than $c''|x-y|^{\alpha}$. 

By using \eqref{*offdist},  the piecewise affine function $\overline{f}$ can be perturbed a bit to obtain 
a  locally non-constant,  piecewise   affine  $c^{-}$-H\"older-$\aaa$ function, still denoted by $\overline{f}$, for which \eqref{*epscl} holds.
 %This concludes the proof.
\end{proof}

%%%%%%%%%%%%%%%%%%%%%%%%%%%%%%%%
%%%%%%%%%%%%%%%%%%%%%%%%%%%%%%%%
%%%%%%%%%%%%%%%%%%%%%%%%%%%%%%%%
%%%%%%%%%%%%%%%%%%%%%%%%%%%%%%%%
%%%%%%%%%%%%%%%%%%%%%%%%%%%%%%%%
%%%%%%%%%%%%%%%%%%%%%%%%%%%%%%%%

\section{Upper bound for $D_*(\alpha, F)$ }\label{*secub}

Our goal now is to prove
Theorem \ref{thm:trivial_upper_bound} which gives
 an upper bound for $D_*(\alpha, F)$ for an arbitrary 
$F\sse \R^{p}$.
The next simple lemma is probably known. 
Since we were unable to find a reference to it we provide its short and simple proof.

\begin{lemma} \label{lemma:slicing}
For any bounded measurable set $F\subseteq \mathbb{R}^{p}$ and $(p-1)$-dimensional hyperplane $L$ with unit normal vector $v$, we have that 
$$\overline{\dim}_B \left((L+tv)\cap F \right) \leq \max\{0,\overline{\dim}_B (F) -1\}$$
for  Lebesgue  almost every $t\in\mathbb{R}$.
\end{lemma}

\begin{proof}
As non-degenerate affine transformations do not change 
the dimension of sets  we can assume that $L$ equals the hyperplane spanned by the first $p-1$ basis vectors
 of the standard basis $(e_i)_{i=1}^{p}$, and $v=e_p$.

Recall Definition \ref{*defboxd} and
let $a_N(F)$ denote the number of $2^{-N}$ grid hypercubes intersected by $F$, and set $s= \max\{1,\overline{\dim}_B (F)\}$.
 Due to  the  definition of the upper box dimension, for every $\eps>0$ there exists $N_0\in\mathbb{N}$ such that for $N>N_0$ we have
\begin{equation} \label{eq:small_intersection}
a_N(F) \leq 2^{(s+\varepsilon)N}.
\end{equation}
For $N>N_0$, define $E_N\subseteq \mathbb{R}$ such that $t\in E_N$ if
\begin{equation} \label{eq:large_intersection}
a_N((L+tv)\cap F) > 2^{(s-1+2\varepsilon)N}.
\end{equation}
We claim that 
\begin{equation} \label{eq:E_N_small}
\lambda(E_N)\leq 2^{-\varepsilon N}.
\end{equation} Indeed, if the reversed inequality holds, then $E_N$ intersects the  interior of  at least $2^{(1-\varepsilon)N}$ grid intervals of length $2^{-N}$, and then by \eqref{eq:large_intersection}, we can deduce
$$a_N(F) > 2^{(s+\varepsilon)N},$$
contradicting \eqref{eq:small_intersection}. Hence \eqref{eq:E_N_small} is justified, which enables us to apply the Borel--Cantelli lemma to the sequence $(E_n)_{n=N_0+1}^{\infty}$. 
It yields that apart from a set of zero measure, for any $t\in\mathbb{R}$ we have
$$a_N((L+tv)\cap F) \leq 2^{(s-1+2\varepsilon)N}$$
for large enough $N$, yielding
$$\overline{\dim}_B \left((L+tv)\cap F \right) \leq  \max\{0,\overline{\dim}_B (F) -1\} + 2\varepsilon$$
for almost every $t$. 
It clearly  gives  the statement of the lemma. 
\end{proof}

\begin{proof}[Proof of Theorem \ref{thm:trivial_upper_bound}]
Every $f\in C^\alpha(F)$ is uniformly continuous on $F$, hence it has a unique continuous extension $f^*$ to $\overline{F}$ (where $\overline{F}$ is the closure of $F$). The function $f^*$ is in $C^\alpha(\overline F)$. Moreover, it is easy to see that $\phi\colon f\mapsto f^*$ is an isomorphism between $C^\alpha(F)$ and $C^\alpha(\overline F)$.
As $\udimb(F)=\udimb(\overline{F})$ and $f^{-1}(r)\subset (f^*)^{-1}(r)$, we can assume that $F$ is closed.

We will prove a stronger statement, notably that for the generic 1-H\"older-$\alpha$ function $f:F\to\mathbb{R}$ and for almost every $r\in\mathbb{R}$ we have
$$\dimh f^{-1}(r)\leq \underline{\dim}_B f^{-1}(r) \leq  \max\{0,\overline{\dim}_B (F) - 1\}.$$
Since the first inequality above is always true we need to verify the second one.
We will calculate these box dimensions by estimating the number of $2^{-N}$ grid cubes intersected by $f^{-1}(r)$, which we denote by $a_N(f, r)$. 
Following this notation, we have
$$\underline{\dim}_B f^{-1}(r) = \underline{\lim}\frac{\log a_N(f, r)}{N\log 2}.$$
(Unless $a_N(f,r)$ is identically zero: in that case, this dimension is simply 0.)

Now for arbitrary $N\in\mathbb{N}$, $\varepsilon>0$, $\delta>0$ denote by $H_N(\varepsilon, \delta)$ the set of 1-H\"older-$\alpha$ functions,
$f$ for which there exists $E\subseteq \mathbb{R}$ with measure $\delta$, such that for any $r\in E$ and for any $m\geq{N}$ we have
$$a_m(f,r)>\left(s+ \varepsilon\right)^m,$$
where
$$s= \max\{1,\exp(\log 2 \cdot (\overline{\dim}_B (F) - 1))\}.$$
For the time being, assume that $H_N(\varepsilon, \delta)$ is nowhere dense for any $N, \varepsilon, \delta$. 
Taking countable union for $\delta=\frac{1}{k}$ shows that 
for $f$ not belonging to a meager set of 1-H\"older-$\alpha$ functions
$$a_m(f,r)>\left(s+ \varepsilon\right)^m$$
holds for any $m\geq{N}$ only in a Lebesgue null-set of $r$s. 
Similarly, taking a countable union for $N\in\mathbb{N}$ shows that 
for $f$ not belonging to a meager set of 1-H\"older-$\alpha$ functions we have that
$$a_m(f,r)\leq \left(s+\varepsilon\right)^m$$
for infinitely many $m$, except for a null-set of $r$s, and hence
$$\frac{\log a_m(f,r)}{m\log 2} \leq \frac{\log(s+\varepsilon)}{\log 2}.$$
However, it immediately yields that for any $\varepsilon>0$, in a residual set
of functions, $f$
$$\underline{\dim}_B f^{-1}(r) \leq \frac{\log(s+\varepsilon)}{\log 2},$$
except for a null-set of $r$s. 
Taking intersection for $\varepsilon=\frac{1}{l}$, $l\in \N$ then yields
$$\underline{\dim}_B f^{-1}(r) \leq \frac{\log s}{\log 2} = \max\{ 0, \overline{\dim}_B (F) - 1\}$$
in a residual set of 1-H\"older-$\alpha$ functions $f$ for almost every $r$,  which is the desired conclusion.
 
Consequently, to complete the proof of this theorem we need to show that $H=H_N(\varepsilon, \delta)$ is nowhere dense for any $N, \varepsilon, \delta$.
 
To this end,
using Lemma \ref{piecewise_affine_approx} fix a family $\mathcal{F}$ of locally non-constant 
piecewise affine 1-H\"older-$\alpha$ functions 
such that they form a dense subset of 1-H\"older-$\alpha$ functions, and fix $N, \varepsilon, \delta$. 
 Now it suffices to prove that any $f_0\in \mathcal{F}$ has a neighborhood $B(f_0, R_0)$ such that for any $f \in B(f_0, R_0)$, we have $f\notin H_N(\varepsilon, \delta)$. 

Assume that $f_0$ has $k$ affine pieces. 
It yields that any level set $f_0^{-1}(r)$ consists of the intersection of $F$ with pieces of at most $k$ hyperplanes. 
These hyperplanes admit only a finite number of different directions, that is they arise as the translation of finitely many fixed $(p-1)$-dimensional hyperplanes. 
Consequently, according to Lemma \ref{lemma:slicing}, the upper box dimension of $f_0^{-1}(r)$ is at most $ \max \{0, \overline{\dim}_B (F) - 1\}$ for almost every $r$. 
It yields that there exists a set $E\subseteq \mathbb{R}$ with 
\begin{equation} \label{eq:exceptional}
\lambda(E)<\frac{\delta}{2}
\end{equation}
and $n_0\in \mathbb{N}$ such that for any $r\notin E$ and $m>n_0$ we have
$$a_m(f_0, r)\leq \left(s+ \varepsilon\right)^m.$$
Fix such an $m>N$.

Now let $\mathcal{H}$ be the family of $2^{-N}$ grid cubes intersected by $F$. 
For any $R>0$, we can define
$$E_1(R) = \bigcup_{T \in  \mathcal{H}}  U_R( f_0(T\cap F))\setminus  f_0(T\cap F).$$
 Since $F$ is compact $ f_0(T\cap F)$ is also compact.  We can fix  a sufficiently small $R>0$  such that for $E_1=E_1(R)$ we have
\begin{equation} \label{eq:new_values}
\lambda(E_1) < \frac{\delta}{2}.
\end{equation}
%as this measure is bounded by $\# \mathcal{H} \cdot 2R$. 
However, if $r\notin E_1$, for any $f\in B(f_0, R)$ we have that $a_m(f_0, r) \ge a_m(f, r)$, as  $f^{-1}(r)\cap T\neq\emptyset$ implies $f_0^{-1}(r)\cap T\neq\emptyset$.  
Putting together \eqref{eq:exceptional} and \eqref{eq:new_values}  we obtain  that for any $f\in B(f_0, R)$, apart from the set
$$E' = E \cup E_1, $$ 
for any $r$ and $f\in B(f_0, R)$ we have
$$a_m(f, r)\leq \left(s+ \varepsilon \right)^m.$$
 Since $ \lambda(E')<\delta$
 
it verifies that $H_N(\varepsilon,\delta)$ is nowhere dense. 
It concludes the proof.
\end{proof}

%%%%%%%%%%%%%%%%%%%%%%%%%%%%%%%%
%%%%%%%%%%%%%%%%%%%%%%%%%%%%%%%%
%%%%%%%%%%%%%%%%%%%%%%%%%%%%%%%%
%%%%%%%%%%%%%%%%%%%%%%%%%%%%%%%%

%%%%%%%%%%%%%%%%%%%%%%%%%%%%%%%%
%%%%%%%%%%%%%%%%%%%%%%%%%%%%%%%%
%%%%%%%%%%%%%%%%%%%%%%%%%%%%%%%%
%%%%%%%%%%%%%%%%%%%%%%%%%%%%%%%%
%%%%%%%%%%%%%%%%%%%%%%%%%%%%%%%%
%%%%%%%%%%%%%%%%%%%%%%%%%%%%%%%%

\section{Computation of $D_*(\alpha,F)$ for an example}\label{*secex}

In this section we prove Theorem \ref{*thex}.

\begin{lemma}\label{lemma 1 d}
If $F\subset\R^2$ is closed, $f\colon\R^2\to\R$ and $\lambda\times\lambda(\{(x,f(x,y)) : (x,y)\in F\}) > 0$, then $D_*^f(F) \ge 1$.
\end{lemma}
\begin{proof}
By Fubini's theorem, there exists a set $H\subset\R$ of positive measure such that for every $r\in H$ we have 
$$
\lambda\{x\in\R : \text{there exists a $y\in\R$ such that } (x,y)\in f^{-1}(r)\} > 0.
$$
That is the projection of $f^{-1}(r)$ onto the $x$ axis is of positive measure, and hence $\dimh\left(f^{-1}(r)\right)\ge 1$.
This implies $D_*^f(F) \ge 1$. 
\end{proof}

\begin{proof}[Proof of Theorem \ref{*thex}]
As $F$ is totally disconnected, its topological dimension is $0$, hence every level set of the typical continuous function defined on it is a singleton, hence $D_*(0,F)=0$ indeed.

Now fix an $\alpha\in (0,1]$. 
The upper estimate $D_*(\alpha,F)\le 1$ is obvious by Theorem \ref{thm:trivial_upper_bound}. 

Using Lemma \ref{piecewise_affine_approx} we can select a countable dense subset $\{f_m : m\in\N\}$ of $C_{1-}^\alpha([0,1/2]^2)$ which consists of locally piecewise affine functions. 
As every $f\in C_1^\alpha(F)$ have an extension in $C_1^\alpha([0,1/2]^2)$, $\{f_m|_{F} : m\in\N\}$ is dense in $C_1^\alpha(F)$.

Next we suppose that $m\in\N$ is fixed. 
Since $\partial_y f_m(x,y)$ takes finitely many different values, we can perturb $f_m$ by adding a function $\tau\cdot y$ with a suitably small $\tau$ to it such that $0<|\partial_y f_m(x,y)|$ wherever $\partial_y f_m(x,y)$ exists.
Thus we can assume that there is a $p_m>0$ such that $p_m<|\partial_y f_m(x,y)|$ (wherever $\partial_y f_m(x,y)$ exists).

Fix $k\ge2$ such that 
\begin{equation}\label{k nagy}
\sum_{l\ge k} 2^{l^2}\cdot \left(2^{-l^3}\right)^\alpha \le \frac{p_m\cdot 2^{-k^2}}{1000} =\frac{p_{m,k}}{1000},
\end{equation}
where $p_{m,k}:=p_m\cdot 2^{-k^2}$.

Since $f_m$ is piecewise affine on $[0,1/2]^2$, we can suppose that $k$ is so large that we can take $j,j'\in\Z$ such that letting $I_i := \left((i-1)\cdot2^{-k^2},(i+1)\cdot 2^{-k^2}\right)$ for $i\in\Z$ the function $f_m$ is affine on $Q_{j,j'} := I_j\times I_{j'}$ and $\lambda(F\cap Q_{j,j'})>0$.

Select a density point $x_0$ of $I_j\cap F_0$.
By our assumptions, $\partial_y f_m(x_0,y)$ takes the same non-zero  value for every $y\in I_{j'}$, and without limiting generality we can assume that it is positive. 
Set $y_0:=(j'-1)2^{-k^2}$ and $y_1:=(j'+1)2^{-k^2}$, that is $I_{j'} = [y_0,y_1]$.
Then 
\begin{equation}\label{y0y1}
f_m(x_0,y_1)-f_m(x_0,y_0) > 2p_{m,k}.
\end{equation}

Let $\delta_m := {p_{m,k}}/{10}$.
Suppose that $f\in B(f_m|_F,\delta_m)\cap C_1^\alpha(F)$.
Denote still by $f$ its $1$-H\"older-$\alpha$ extension to $[0,1/2]^2$.
By \eqref{y0y1}, 
\begin{equation*}
f(x_0,y_1)-f(x_0,y_0) > p_{  m,k }.
\end{equation*}
Since $x_0$ is a density point of $F_0$ and $f\in C_1^\alpha([0,1/2]^2)$ we can choose $\delta_0>0$ such that 
$\lambda(F_0\cap [x_0,x_0+\delta_0]) > 0.99\delta_0$ and $|f(x_0,y_i)-f(x,y_i)| \le 0.01 p_{m,k}$ for $x\in[x_0,x_0+\delta]$ and $i=0,1$.
This implies 
$$
[f(x_0,y_0)+0.01p_{m,k},f(x_0,y_1)-0.01p_{m,k}]\subset \{f(x,t) : t\in[y_0,y_1]\}
\text{ for $x\in[x_0,x_0+\delta]$,}
$$
and by \eqref{y0y1} we have $f(x,y_1)-f(x,y_0)>0.98 p_{  m,k }$.
Thus, 
\begin{equation*}%{k nagy}
\begin{gathered}
\lambda\big(\{f(x,y) : y\in[y_0,y_1]\cap F_0\}\big) \\
\ge \lambda\big(\{f(x,y) : y\in[y_0,y_1]\}\big) - \lambda\big(\{f(x,y) : y\in[y_0,y_1]\setminus F_0\}\big)
\end{gathered}
\end{equation*}
(using the definition of $F_0$ and $f\in C_1^\alpha([0,1/2]^2)$ we can estimate the jumps on the intervals contiguous to $F_{0}$)
\begin{equation*}%{k nagy}
\begin{gathered}
\ge 0.98 p_{m,k} -2 \sum_{l=k}^\infty 2^{l^2}\left(2^{-l^3}\right)^\alpha 
\underset{\text{by \eqref{k nagy}}}\ge 0.98 p_{m,k} -\frac{p_{m,k}}{500} > 0.9 p_{m,k}.
\end{gathered}
\end{equation*}

By Fubini's theorem 
\begin{align*}
0 &< \lambda\times\lambda\big(\big\{(x,f(x,y)): (x,y)\in\big([x_0,x_0+\delta_0]\cap F_0\big)\times\big([y_0,y_1]\cap F_0\big)\}\big) \\
&\le \lambda\times\lambda\big(\big\{(x,f(x,y)): (x,y)\in F\big\}\big).
\end{align*}
According to Lemma \ref{lemma 1 d}, this implies $D_*^f(F)\ge 1$.
Put $\cag = \bigcup_{m=1}^\infty B(f_m,\delta_m)\cap C_1^\alpha(F)$.
Then $\cag$ is an dense open subset of $C_1^\alpha(F)$ and for every $f\in\cag$ we have $D_*^f(F)\ge1$.

Therefore $D_*(\alpha,F)\ge 1$.
Since we also know that $D_*(\alpha,F)\le 1$, this completes the proof.
\end{proof}

\section{Dense $G_{\ddd}$ sets in which $D_*^f(F) = D_*(\alpha,F)$ for any $f$ }\label{*secdgd}

%\input{sierthmgenfix}

%%%%%%%%%%%%%%%%%%%%%%%%%%%%%%%%
%%%%%%%%%%%%%%%%%%%%%%%%%%%%%%%%
%%%%%%%%%%%%%%%%%%%%%%%%%%%%%%%%
%%%%%%%%%%%%%%%%%%%%%%%%%%%%%%%%

\begin{lemma}\label{*lemdfb}
Suppose that   $0< \aaa\leq 1$,   $F\subset\R^p$ is compact, $E\subset\R^p$ is open or closed, and $\cau\subset C_1^\alpha(F)$ is open.
If $\{f_1,f_2,\ldots\}$ is a countable dense subset of $\cau$, then there is a dense $G_\delta$ subset $\cag$ of $\cau$ such that 
\begin{equation}\label{eqdfb}
\sup_{f\in\cag} D_*^f(F\cap E) \le \sup_{k\in\N} D_*^{f_k}(F\cap E).
\end{equation}
\end{lemma}

\begin{proof}
First we assume that $E$ is closed.
We can suppose that $E\subset F$.

Since countable union of sets of measure zero is still of measure zero we can choose
a set $R_{0}\sse \R$ such that $\lll(\R\sm R_{0})=0$ and  for any $k$ 
\begin{equation}\label{*lrdfkr}
D^{f_{k}}(r,E)\le  \sup_{k'\in\N} D_*^{f_{k'}}(E) \text{ for any $r\in R_{0}$.}
\end{equation}

Suppose that $D_1>\sup_{k\in\N} D_*^{f_k}(E)$, and fix $k\in \N$ and $r\in R_{0}$.
Recall \eqref{*defdimh}. 
For every $\delta>0$ there exists   $ \{ U_{j,k,r} \}_{j=1}^{\oo}$, a  $\ddd\text{-cover of }f_{k}^{-1}(r)\cap E$ such that $\sum_{j}|U_{j,k,r}|^{D_1}<1 $.
As we remarked after \eqref{*defdimh} we can assume that the sets $U_{j,k,r}$ are open.  

Next we suppose that $k,n\in\N$ are fixed and for $r\in R_0$
we consider $\ddd=\frac{1}{n}$-covers, $\{ U_{j,k,r,n}\}$ of $f_{k}^{-1}(r)\cap E$. Of course, if $f_{k}^{-1}(r)\cap E$ is empty then it may happen that these covers are also empty.
As $E\sm \bigcup_{j}U_{j,k,r,n}$ is compact, $f_{k}$ is continuous
and $f_{k}(x)\not = r$ for every $x\in E\sm \bigcup_{j}U_{j,k,r,n}$, we have
\begin{equation}\label{rho_{k,n,r}}
0<\rrr_{k,n,r}:=\min\Big \{ 1, \inf \{ |f_{k}(x)-r|: x\in E\sm \bigcup_{j}U_{j,k,r,n} \} \Big \} \text{ for any }
r\in R_0
\end{equation}
(where the infimum of the empty set is $+\infty$ by convention).
Since $f_k$ is continuous, $f_{k}(F)$ is bounded.  Hence  
we can choose $\bbM_k$ such that
\begin{equation}\label{*ref1}
\text{$f_{k}^{-1}(r)\cap E=\ess$,
if $r\not \in (-\bbM_k+1,\bbM_k-1)$.
}
\end{equation}
Choose a compact subset 
\begin{equation}\label{*ref2}
\text{$\bbR_{k,n}\sse R_0\cap (-\bbM_k,\bbM_k)$
such that $\lll(\bbR_{k,n})>2\bbM_k-2^{-n}$.}
\end{equation}
Then we can choose a finite subset $\car_{k,n}\sse  \bbR_{k,n}$, such that
$\bbR_{k,n}\sse \bigcup_{r\in \car_{k,n}}(r-\rrr_{k,n,r},r+\rrr_{k,n,r})$.
Moreover, the compactness of $\bbR_{k,n}$ also yields that we can choose $\rho_{k,n}\in (0,1) $ such that
for any $r\in \bbR_{k,n}$
we can find $\bbr_{k,n}(r)\in  \car_{k,n}$
such that 
\begin{equation}\label{rho_{k,n}}
(r-\rho_{k,n},r+\rho_{k,n})\sse (\bbr_{k,n}(r)-\rrr_{k,n,\bbr_{k,n}(r)},\bbr_{k,n}(r)+\rrr_{k,n,\bbr_{k,n}(r)}).
\end{equation}

Let $\cag_{n}=\bigcup_{k}B(f_{k},\rrr_{k,n})\cap \cau$ and
$\cag =\bigcap_{n}\cag_{n}.$

Suppose $f\in \cag.$ Then  there exists a sequence $k_{n}$  such that $f\in B(f_{k_{n}},\rrr_{k_{n},n})$ for every $n$.
  
Set $R_\infty:=\bigcap_{m}\bigcup_{n\geq m} (\bbR_{k_n,n}\cup(\R\sm (-\bbM_{k_n},\bbM_{k_n})))$. 
By \eqref{*ref1}, \eqref{*ref2} and the Borel--Cantelli lemma, $\lambda(\R\setminus R_\infty)=0$, and for every $r\in R_\infty$ 
either $f^{-1}(r)\cap E=\ess$ or for infinitely many $n$ 
\begin{equation*}
\begin{split}
f^{-1}(r)
&\underset{\hphantom{\eqref{rho_{k,n}}}}\subset f_{k_n}^{-1}\big((r-\rho_{k_n,n},r+\rho_{k_n,n})\big)\cap E \\
&\underset{\eqref{rho_{k,n}}}\subset 
f_{k_n}^{-1}\big((\bbr_{k_n,n}(r)-\rho_{k_n,n,\bbr_{k_n,n}(r)},\bbr_{k_n,n}(r)+\rho_{k_n,n,\bbr_{k_n,n}(r)})\big) \cap E
\underset{\eqref{rho_{k,n,r}}}\subset \bigcup_j U_{j,k_{n},\bbr_{k_{n},n}(r),n},
\end{split}
\end{equation*} 
that is, the system $\{U_{j,k_{n},\bbr_{k_{n},n}(r),n}\}$ is a $\frac1n$-cover of $f^{-1}(r)\cap E$.
Thus, using the inequality $\sum_{j}|U_{j,k_{n},\bbr_{k_{n},n}(r),n}|^{D_1}<1$, we obtain  $\dim_{H}(f^{-1}(r)\cap E) \le D_1$ for a.e. $r\in \R$, and hence $D^{f}(r,F)\leq D_1$.
As $D_1>\sup_{k\in\N} D_*^{f_k}(E\cap F)$ was chosen arbitrarily \eqref{eqdfb} is satisfied.

Now suppose that $E$ is open.
For every $n\in\N$ set 
$$
E_n := \Big \{x\in E \cap F : \inf\{|x-y| : y\in F\setminus E\} \ge \frac1n\Big\}.
$$
Observe that $E_n \subseteq F $ is closed. We can apply the previously proved case to $E_n$. 
We obtain a dense $G_\delta$ subset $\cag'_n$ of $\cau$ such that $\sup_{f\in\cag'_n}D_*^f(E_n) \le \sup_{k\in\N} D_*^{f_k}(E_n)$.
Let $\cag:=\bigcap_{n=1}^\infty\cag'_n$.
If $f\in\cag$ then 
$$
D_*^f(F\cap E) 
=  \sup_{n\in\N} D_*^{f}(E_n)
\le \sup_{n\in\N} \sup_{k\in\N} D_*^{f_k}(E_n) 
\leq \sup_{k\in\N} D_*^{f_k}(F\cap E).
$$
\end{proof}

\begin{proof}[Proof of Theorem \ref{thm:generic_existence}]
Let $D_0 := D_*(\alpha,F)$.

For every $k\in\N$ choose a $\cag^k\in\mg_{1,\alpha}(F)$ for which $D_0-\frac1k\le \inf_{f\in\cag^k}D_*^f(F)$. 
Set $\cag_0=\bigcap_{k=1}^\infty \cag^k$. 
We have that $\cag_0\in\mg_{1,\alpha}(F)$ and $D_0\le \inf_{f\in\cag}D_*^f(F)$.
It is enough to prove that for every $k\in\N$ there is a $\cag_k\in\mg_{1,\alpha}$ such that 
\begin{equation}\label{D_*^f kicsi}
\sup_{f\in\cag_k}D_*^f(F) \le D_k :=D_0+\frac1k ,
\end{equation}
since then $\cag:=\bigcap_{k=0}^\infty \cag_k$ is a proper choice.

Fix $k\in\N$. %and set $D_k:=D_0+\frac1k$.

The set $H_{k} := \{f\in\cag_0 : D_*^f(F)\le D_k\}$ cannot be nowhere dense in $C_1^\alpha(F)$, since otherwise $\cag':=\cag_0\setminus \cl(H_{k})$ would be in $\mg_{1,\alpha}(F)$ and it would hold that 
$$
\inf_{f\in\cag'}D_*^f(F)  \ge D_k > D_0 = D_*(\alpha,F),
$$ 
which contradicts the definition of $D_*(\alpha,F)$.
Hence we can take $f_1\in C_1^\alpha(F)$ and $\delta_1>0$ such that $H_{k}$ is dense in $B(f_1,\delta_1)\cap C_1^\alpha(F)$.
Choose a $\delta_2>0$ to satisfy $\delta_2^\alpha \le  {\delta_1}/{64}$.
As $F$ is compact, we can take a finite set $A\subset F$ such that $\bigcup_{a\in A}B(a,\delta_2)$ covers $F$.

Suppose that $a$ is fixed, $\eps>0$ and $g_0\in C_1^\alpha(F)$ is an arbitrary function.
Let $E:=B(a,\delta_2)\cap F$.
By the H\"older property, for every $f\in C_1^\alpha(F)$
$$
\diam\big(f(E)\big)\le (2\delta_2)^\alpha \le \frac{\delta_1}{32}.
$$
Thus setting %$g_1:= g_0|_E-g_0(a)+f_1(a)$ 
$$
g_1(x) := \min\big\{\max\{g_0(x)-g_0(a)+f_1(a),f_1(x)-\delta_1/2\},f_1(x)+\delta_1/2\big\}.
$$
we obtain 
\begin{equation}\label{g_1}
g_1|_E=g_0|_E-g_0(a)+f_1(a)
\end{equation}
(since $g_1(a)=f_1(a)$ and $\diam\big(g_0(E)\big)+\diam\big(f_1(E)\big)\le  {\delta_1}/{16}$).
As $g_1\in B(f_1,\delta_1)$, 
 we can take $g_2\in H_k$ such that $\big|\big|g_1|_E-g_2|_E\big|\big|<\eps/100$.
Set 
$$
g_3(x) := \min\big\{\max\{g_2(x)-g_2(a)+g_0(a),g_0(x)-\eps\},g_0(x)+\eps\big\}.
$$
Obviously $\big|\big|g_3-g_0\big|\big|\le\eps$.
By \eqref{g_1}  and by the  definition of $g_2$, for every $x\in  E$
\begin{equation}\label{g_1_cau-ban}
\begin{gathered}
\big|g_2(x)-g_2(a)+g_0(a)-g_0(x)\big| \\
\le \big|g_2(x)-g_1(x))\big| + \big|g_1(a)-g_2(a))\big| + \big|g_0(a)-g_1(a) + g_1(x)-g_0(x)\big| \\
\le \eps/100+\eps/100+0 <\eps,
\end{gathered}
\end{equation}
hence $g_3(x) = g_2(x)-g_2(a)+g_0(a)$ for every $x\in  E$.
Thus $D_*^{g_3}( E) = D_*^{g_2}( E) \le D_k$ since $g_2\in H_k$.

To sum up, for every $g_0\in C_{1}^{\alpha}(F)$, $a\in A$ and $\eps>0$ we can find a $g_3\in C_{1}^{\alpha}(F)$ such that $||g_0-g_3||\le\eps$ and $D_*^{g_3}\big(B(a,\delta_2)\cap F\big) \le D_k$.
Consequently, by Lemma \ref{*lemdfb} for every $a\in A$ there is a $\cag_a^k\in\mg_{1,\alpha}$ satisfying 
$$\sup_{f\in\cag_a^k} D_*^f(B(a,\delta_2) \cap F) \le D_k.$$
Then \eqref{D_*^f kicsi} is true for $\cag^k := \bigcap_{a\in A} \cag_a^k$, which completes the proof.
\end{proof}

%%%%%%%%%%%%%%%%%%%%%%%%%%%%%%%%
%%%%%%%%%%%%%%%%%%%%%%%%%%%%%%%%
%%%%%%%%%%%%%%%%%%%%%%%%%%%%%%%%
%%%%%%%%%%%%%%%%%%%%%%%%%%%%%%%%
%%%%%%%%%%%%%%%%%%%%%%%%%%%%%%%%
%%%%%%%%%%%%%%%%%%%%%%%%%%%%%%%%

\section{Monotonicity of $D_*(\alpha,F)$ in $\alpha$}\label{*secmon}

%%%%%%%%%%%%%%%%%%%%%%%%%%%%%%%%
%%%%%%%%%%%%%%%%%%%%%%%%%%%%%%%%
%%%%%%%%%%%%%%%%%%%%%%%%%%%%%%%%
%%%%%%%%%%%%%%%%%%%%%%%%%%%%%%%%

%\input{sierthdmon}

\begin{proof}[Proof of Theorem \ref{thm:generic_monotonicity}]
Suppose that $\alpha'>\alpha>0$. If $C^{\alpha'}_1 (F)$ was dense in $C^{\alpha}_1(F)$, we could rely on the generic function in $C^{\alpha'}_1 (F)$ determining $D_*(\alpha',F)$ to obtain conclusions about $D_*(\alpha,F)$ in a rather standard way. However, it is not the case, which raises certain technical difficulties in connecting these function spaces. We handle it as follows.

The set $C^{\alpha'}(F) \cap C^{\alpha}_{1-}(F)$ is dense in the separable space $C^{\alpha}_1(F)$. Hence we can  select  a sequence 
$$(f_{k,1})_{k=1}^{\infty}\subseteq C^{\alpha'}(F) \cap C^{\alpha}_{1-}(F)$$
dense in $C^{\alpha}_1(F)$. 
Due to the two parts of this containment, we can find some $M_{k,1}>0$ and $0<c_{k,1}<1$, $k=1,2,...$ such that
$$f_{k,1}\in C^{\alpha'}_{M_{k,1}}(F) \cap C^{\alpha}_{c_{k,1}}(F)\text{ holds for $k=1,2,...$.}$$
Consequently, $\frac{1}{M_{k,1}}f_{k,1}\in C^{\alpha'}_{1}(F)$. Now due to Theorem \ref{thm:generic_existence}, there exists a dense $G_\delta$ set $\mathcal{G}_0\subseteq C^{\alpha'}_1(F)$ such that for any $f\in\mathcal{G}_0$ we have $D_{*}^f(F)=D_{*}(\alpha', F)$. This observation immediately yields the existence of a sequence $(f_{k,2})_{k=1}^{\infty}\in \mathcal{G}_0$ such that
\begin{equation} \label{eq:delta_k_def}
\left\|\frac{1}{M_{k,1}}f_{k,1} - f_{k,2}\right\|< \delta_k
\end{equation}
for some $\delta_k$ to be fixed later. By applying a simple rescaling, let 
$$f_{k,3} := {M_{k,1}}f_{k,2} \in C^{\alpha'}_{M_{k,1}}(F).$$
For any $k$ from $f_{k,2}\in \mathcal{G}_0$ it follows that $D_{*}^{f_{k,3}}(F)=D_{*}(\alpha', F)$. 

Now let us set $c_{k,2} = \frac{1+c_{k,1}}{2}\in (c_{k,1}, 1)$. Our claim is that for some well-chosen $\delta_k$, we have $f_{k,3}\in C^{\alpha}_{c_{k,2}}(F)$ as well. Momentarily assume that this claim holds. Then the proof can be concluded swiftly: by \eqref{eq:delta_k_def}, we have
\begin{equation} \label{eq:delta_k_deff}
\left\|f_{k,1} - f_{k,3}\right\|< \delta_k M_{k,1}<\frac{1}{k},
\end{equation}
where the second inequality can be guaranteed by the choice of $\delta_k$. This implies that the sequence $(f_{k,3})_{k=1}^{\infty}$ is dense in $C^{\alpha}_1(F)$ as well. Now to this sequence we can apply Lemma \ref{*lemdfb} with the roles $E=\mathbb{R}^p$ and $G=C^{\alpha}_1(F)$ to obtain a dense $G_\delta$ set $\mathcal{G}\subseteq C^{\alpha}_1(F)$ such that for any $f\in \mathcal{G}$ we have
$$D_{*}(\alpha, F) \le \sup_{f\in\cag} D_*^f(F) \le \sup_{k\in\N} D_*^{f_{k,3}}(F) = D_{*}(\alpha', F),$$
where the  second  inequality follows from the lemma, while the equality follows from 
the construction of the sequence $(f_{k,3})_{k=1}^{\infty}$. Altogether we obtain the statement of the theorem indeed.

It only remains to prove the above claim, that is for any $x,y\in F$ and $f=f_{k,3}$ we have
$$|f(x)-f(y)|\leq c_{k,2}|x-y|^{\alpha}.$$
We use the standard technique of separating two cases based on the distance $|x-y|$. Notably, assume first that
\begin{equation} \label{eq:small_distance_generic_mon}
|x-y|\leq \left(\frac{c_{k,2}}{M_{k,1}}\right)^{\frac{1}{\alpha'-\alpha}}.
\end{equation}
Then due to $f=f_{k,3}\in C^{\alpha'}_{M_{k,1}}(F)$ we have
$$|f(x)-f(y)|\leq M_{k,1}|x-y|^{\alpha'}=M_{k,1}|x-y|^{\alpha'-\alpha}|x-y|^\alpha\leq c_{k,2}|x-y|^\alpha,$$
where the last inequality directly follows from \eqref{eq:small_distance_generic_mon}. 

Now assume the opposite inequality concerning the distance $|x-y|$, that is
\begin{equation} \label{eq:large_distance_generic_mon}
|x-y|> \left(\frac{c_{k,2}}{M_{k,1}}\right)^{\frac{1}{\alpha'-\alpha}}.
\end{equation}
In this case, we appropriately substitute $f=f_{k,3}$ by $g=f_{k,1}$ and rely on $g\in C^{\alpha}_{c_{k,1}}(F)$. Notably,
$$|f(x)-f(y)|\leq |f(x)-g(x)|+|g(x)-g(y)|+|f(y)-g(y)|< c_{k,1}|x-y|^{\alpha} + 2\delta_k M_{k,1}$$ $$\leq c_{k,2}|x-y|^{\alpha},$$
where due to \eqref{eq:large_distance_generic_mon}, the last inequality follows from
$$2\delta_k M_{k,1}\leq (c_{k,2}-c_{k,1})\left(\frac{c_{k,2}}{M_{k,1}}\right)^{\frac{\alpha}{\alpha'-\alpha}},$$
which simply poses another restriction on the choice of $\delta_k$. Consequently, if $\delta_k$ satisfies this, then $f\in C^{\alpha}_{c_{k,2}}(F)$ indeed, and if the assumption \eqref{eq:delta_k_def} holds as well, then the concluding step of the proof is also valid. 
\end{proof}

%%%%%%%%%%%%%%%%%%%%%%%%%%%%%%%%
%%%%%%%%%%%%%%%%%%%%%%%%%%%%%%%%
%%%%%%%%%%%%%%%%%%%%%%%%%%%%%%%%
%%%%%%%%%%%%%%%%%%%%%%%%%%%%%%%%
%%%%%%%%%%%%%%%%%%%%%%%%%%%%%%%%
%%%%%%%%%%%%%%%%%%%%%%%%%%%%%%%%

\section{Self similar sets and $D_*(\aaa,F)$}\label{secd*aaa}

In  this section we prove Theorem \ref{*thessup}.

\begin{comment}
\begin{remark} \label{remark:selfsimilar_Lip}
Certainly, the case $\alpha=1$ bears interest, especially if we consider that most of our results did not exclude this value. Nevertheless it is a question we cannot answer now: the proof of Theorem \ref{*thessup} will rely on Proposition \ref{*Ddo}, which we could not generalize for $\alpha=1$.
\end{remark}
\end{comment}

Since generic continuous functions are non-constant on sets consisting of more than two points, for connected $F$s containing at least two points the range of the generic continuous function
contains an interval and hence is of positive Lebesgue measure.

 As we mentioned in the introduction if $F$ is the disjoint union of two fractals
$F_{1}$ and $F_{2}$, with $D_{*}(\aaa, F_1)<D_{*}(\aaa, F_2)$ then it is easy to see that it is not necessarily true that for the generic $1$-H\"older-$\aaa$  function 
$D_{*}(\aaa, F)$ equals the Hausdorff dimension of almost every level set in the range of the function. 

Indeed, suppose that we put a scaled copy $S$ of the Sierpi\'nski triangle into $[0,1/4]\times [0,1/4]$, and $T$ denotes $ [1/2,3/4]\times \{ 0 \}$.
Put $F=S\cup T$. Suppose that $f(x,y)$ is a function which is constant $0$
 on $S$ and equals $1/8+x$ on $T$. 
 Results of  \cite{sierc} imply that $D_{*}(1/2, S)>0$ and
 $D_{*}(1/2, T)=0$  by Theorem  \ref{thm:trivial_upper_bound}.
 Then for some  generic $1$-H\"older-$1/2$ function $g$ in the ball $B(f,1/16)$  for almost every $r<1/16$   with $r\in g(S)$    we have $\dimh g^{-1}(r)= D_{*}(1/2, S)>0$
 and   for almost every $r>1/16$  we have $\dimh g^{-1}(r)=0$. 
As $g(S)\subset (-\infty,1/16)$ and $g(T)\subset (1/16,\infty)$, and $\lambda(g(S))>0$ and $\lambda(g(T))>0$ for a generic $g$,  this counterexample is valid.

We put $\DDD:=D_*(\aaa,F)$ and
\begin{equation}\label{*defkfd}
\kkk(f,\ddd):=\frac{\lll\{ r\in f(F):\dimh f^{-1}(r)> \DDD-\ddd \}}{\lll(f(F))}.
\end{equation}

The strategy of the proof of Theorem \ref{*thessup} is the following.
First we reduce it to Proposition \ref{*Ddo}.
In Lemma \ref{*clggd} we show that if we have a dense set of functions with
relatively small portion of level sets with Hausdorff dimension close to $\DDD$ then there is a dense $G_{\ddd}$ set of functions with the same property.
Based on this lemma in Proposition \ref{*ko} we show that for any $\ddd_{0}>0$
we can find a $\kkk_{0}>0$ and an open ball in $C_{1}^{\aaa}(F)$
such that for any function $f$ from this ball at least $\kkk_{0}$ portion of the range 
corresponds to level sets with Hausdorff dimension larger than $\DDD-\ddd_{0}$.
In the proof of Proposition \ref{*Ddo} we use rescaled (both in range and domain) affine versions of the functions from the ball in  Proposition \ref{*ko}. This way we obtain functions 
for which uniformly in any sufficiently large interval in the range of the function 
a portion of the range 
corresponds to level sets with Hausdorff dimension larger than $\DDD-\ddd_{0}$.
Finally, Lebesgue's density theorem will yield that almost every 
level set is of Hausdorff dimension larger than $\DDD-\ddd_{0}$ for functions in a dense $G_{\ddd}$ set. This will complete the proof of Proposition \ref{*Ddo}.

\begin{proposition}\label{*Ddo}
Suppose that $F$ is a connected self-similar set and $0<\alpha<1$. Then  for every $\ddd_{0}>0$  there exists a dense $G_{\ddd}$ set $\cag$
in $C_{1}^{\aaa}(F)$  such that 
for every $f\in \cag$, 
\begin{equation}\label{*Ddoeq}
\dimh f^{-1}(r)\geq \DDD-\ddd_{0} \quad\text{ for a.e. }r\in f(F).
\end{equation}
\end{proposition}

We prove this later. Using this proposition it is very easy to prove Theorem \ref{*thessup}.

\begin{proof}[Proof of Theorem \ref{*thessup} based on Proposition \ref{*Ddo}]

Using Theorem \ref{thm:generic_existence} choose a dense $G_{\ddd}$ set $\cag_{0}$ such that $D_{*}^{f}(F)=D_{*}(\aaa,F)=\DDD$ for any $f\in \cag_0$.  This implies that if $f\in \cag_{0}$ then $\dimh f^{-1}(r) \leq \DDD$ for a.e. $r\in f(F)$.

For $\ddd_{0}=1/n$, $n\in \N$ select $\cag_{n}$ by using Proposition  \ref{*Ddo}
and set $\cag=\bigcap_{n=0}^{\oo}\cag_{n}$.
Then for every $f\in \cag$ we have $\dimh f^{-1}(r) = \DDD$ for a.e. $r\in f(F)$.
\end{proof}

Before proving Proposition \ref{*Ddo} we need the next lemma which is followed by Proposition \ref{*ko}.

\begin{lemma}\label{*clggd}
Suppose that $0<\kkk_{0}<1$  and there exists $\ddd_{0}>0$  such that  one can select 
a dense set $f_{n}\in C_{1}^{\aaa}(F)$ for which $\kkk(f_{n},\ddd_{0})<\kkk_{0}$.
Then there exists a dense $G_{\ddd}$ set $\cag^{\kkk_{0}}$  such that $\kkk(f,\ddd_{0})\leq \kkk_{0}$ for every $f\in \cag^{\kkk_{0}}$.
\end{lemma}

\begin{proof}%[Proof of Claim \ref{*clggd}]
Given $k\in\N $ using our dense set we will select radii
$\ddd_{n,k}$. We will define $\cag_{k}=\bigcup_{n}B(f_{n},\ddd_{n,k})$ and
$\cag^{\kkk_{0}}=\bigcap_{k }\cag_{k}.$

Suppose that $n$ and $k$ are given. 
Set $$\mathbf{H}_n=\{ r\in f_{n}(F):\dimh f_n ^{-1}(r)\leq \DDD-\ddd_{0}\}.$$
By assumption $\kkk(f_{n},\ddd_{0})<\kkk_{0}$ and hence
$$\lll(\mathbf{H}_n)>(1-\kkk_{0})\lll (f_{n}(F)).$$
Select a compact set $\GGG_{n}\sse \mathbf{H}_n$  such that 
\begin{equation}\label{*Gnkdef}
\lll(\GGG_{n})>(1-\kkk_{0})\lll(f_{n}(F)).
\end{equation} 
Using the definition of the Hausdorff dimension for every $r\in\GGG_{n}$ we select open sets $U_{n,k,r,j}$
 such that 
 $f^{-1}_{n}(r)\sse \bigcup_{j} U_{n,k,r,j}$, $|U_{n,k,r,j}|<1/k$
 and
 \begin{equation}\label{*IIa}
 \sum_{j}|U_{n,k,r,j}|^{\DDD-\ddd_{0}+\frac{1}{k}}<1.
 \end{equation}
Put 
$$\rrr(n,k,r):=\min \{ |f_{n}(x)-r|: x\in F\sm \bigcup_{j}U_{n,k,r,j} \}>0,$$
where the last inequality holds due to the compactness of $F$.
Since $\GGG_{n}$ is also compact we can select finitely many $r_{l}\in \GGG_{n} $
 such that 
 $$\GGG_{n}\sse \bigcup_{l}\Big(r_{l}-\frac{\rrr(n,k,r_{l})}{2},r_{l}+\frac{\rrr(n,k,r_{l})}{2}\Big).$$
 Let $$\ddd_{n,k}=\min\Big \{ \frac{1}{n+k}, \min_{l}\big\{  \frac{\rrr(n,k,r_{l})}2\big \} \Big \}>0.$$
 
 Suppose that $f\in B(f_{n},\ddd_{n,k})$ and $r\in\GGG_{n}$. Then there exists an $l$  such that 
 $r\in (r_{l}-\frac{\rrr(n,k,r_{l})}{2},r_{l}+\frac{\rrr(n,k,r_{l})}{2})$.
 Suppose that $x\in f^{-1}(r)$.
 Then $$f_{n}(x)\in (r-\ddd_{n,k},r+\ddd_{n,k})\sse (r_{l}-\rrr(n,k,r_{l}),r_{l}+\rrr(n,k,r_{l})).$$
 Therefore $x\in \bigcup_{j}U_{n,k,r_{l},j}$ and
 \begin{equation}\label{*IIIa}
 f^{-1}(r)\sse \bigcup_{j}U_{n,k,r_{l},j}.
 \end{equation}
 
 Suppose that $f\in \cag^{\kkk_{0}}$. Then  there exists a sequence $n(k)$
  such that $f\in \bigcap_{k}B(f_{n(k)}, \ddd_{n(k),k})$.
 It is also clear that  $\lim_{k\to\infty} \lambda\left(f_{n(k)}(F)\triangle f(F)\right)=0$.
 We have
 $$\lll(\GGG_{n(k)})> (1-\kkk_{0})\lll(f_{n( k)}(F)).$$
 Let $\GGG_{f}:=\bigcap_{j=1}^{\oo}\bigcup_{k=j}^{\oo}\GGG_{n(k)}$.
 Then $\lll(\GGG_{f})\geq (1- \kkk_{0})\lll(f(F))$.
Suppose that $r\in \GGG_{f}.$

From \eqref{*IIa} and \eqref{*IIIa} we infer that
$$\dimh f^{-1}(r)\leq \DDD-\ddd_{0}.$$
This implies that 
$$\lll\{ r\in f(F):\dimh f^{-1}(r)> \DDD-\ddd_0 \}\leq \kkk_{0}\lll(f(F)),$$
that is $\kkk(f,\ddd_{0})\leq \kkk_{0.}$  
\end{proof}

 In the sequel we will take balls in $C_{1}^{\aaa}(F)$ and hence, for ease of notation we will consider balls in this space, that is for example we will write $ B(f_{0},\rrr_{0})$ instead of $ B(f_{0},\rrr_{0})\cap C_{1}^{\aaa}(F)$.

\begin{proposition}\label{*ko}
For every $\ddd_{0}>0$ there exist $0<\kkk_{0}\leq 1$, $f_{0}\in C_{1}^{\aaa}(F)$, and
$\rrr_{0}>0$  such that 
\begin{equation}\label{*eqko}
\kkk(f,\ddd_{0})\geq \kkk_{0}\quad \text{ for every }f\in B(f_{0},\rrr_{0}).
\end{equation}
\end{proposition}

\begin{proof}
Proceeding towards a contradiction suppose that the statement of 
the proposition is not true. 
Then  there exists $\ddd_{0}$  such that for every $0<\kkk_{0}\leq 1$ one can select a dense set $f_{n}\in C_{1}^{\aaa}(F)$ such that $\kkk(f_{n},\ddd_{0})<\kkk_{0}$.
For $\kkk_{0,n}=1/n$ 
use Lemma \ref{*clggd} and
take the dense $G_{\ddd}$ sets, $\cag^{\kkk_{0,n}}$  such that $\kkk(f,\ddd_{0})<\kkk_{0,n}$ for every $f\in \cag^{\kkk_{0,n}}$. 

Let $\cag^{0}=\bigcap_{n=1}^{\oo} \cag^{\kkk_{0,n}}$. It is also dense $G_{\ddd}$. Suppose that $f\in \cag^{0}$. Then
$$\lll\{ r\in f(F):\dimh f^{-1}(r)> \DDD-\ddd_{0} \}\leq \kkk_{0,n}\lll(f(F))\text{ for all }n.$$
This implies $\lll\{ r\in f(F):\dimh f^{-1}(r)\geq \DDD-\ddd_{0} \}=0$, but
$\DDD-\ddd_{0}<\DDD=D_*(\aaa,F)$ and this contradicts the definition of $D_*(\aaa,F)$.
\end{proof}

Now we are ready to prove Proposition \ref{*Ddo}.

\begin{proof}[Proof of Proposition \ref{*Ddo}]
Without limiting generality we can suppose  that $|F|=1.$
By using Lemma \ref{piecewise_affine_approx} select 
a dense set $\{ f_{n}\}$ in $C_{1}^{\aaa}(F)$  consisting of locally non-constant piecewise affine $1^{-}$-H\"older-$\alpha$ functions.
Since $F$ is connected $f_{n}(F)=[m_{n},M_{n}]$ with $m_{n}<M_{n}$.
Since $f_{n}$ is piecewise affine, it is Lipschitz-$K_{n}$. 
Without limiting generality we assume that $K_{n}\geq 1$.
Since it is $1^{-}$-H\"older-$\alpha$  it is $c_{n}$-H\"older-$\aaa$
with a $c_{n}<1$.
We will select a sufficiently large $L_{n,k}>(n+k)(M_{n}-m_{n}+1)$.
Set 
$$\mathbf{p}_{n,k}(t)=m_{n}+t\cdot \frac{M_{n}-m_n}{L_{n,k}},\quad t=0,...,L_{n,k}-1.$$

For each $t$ choose $\mathbf{x}(t) \in F$ such that $f_{n}(\mathbf{x}(t))=\mathbf{p}_{n,k}(t)$.
If $t\not=t'$ then
$$K_{n}\geq K_{n} |\mathbf{x}(t)-\mathbf{x}(t')|\geq |f_{n}(\mathbf{x}(t))-f_{n}(\mathbf{x}(t'))|\geq \frac{M_{n}-m_n}{L_{n,k}}$$
implies 
\begin{equation}\label{*VI*a}
|\mathbf{x}(t)-\mathbf{x}(t')|\geq \frac{M_{n}-m_n}{L_{n,k}K_n}.
\end{equation}

By using self similarity of $F$ and $1/\aaa>1$ select a sufficiently large $L_{n,k}$ and a similarity $\FFF_{t}$  such that 
$\mathbf{x}(t)\in \FFF_{t}(F)$ and
\begin{equation}\label{*VI*b}
 \frac{M_{n}-m_n}{3L_{n,k}K_n}  >
\Big (\frac{M_{n}-m_n}{3L_{n,k}} \Big )^{1/\aaa}
\geq |\FFF_{t}(F)|>
\mathbf{q}_{\min} \Big (\frac{M_{n}-m_n}{3L_{n,k}} \Big )^{1/\aaa}.
\end{equation}
Observe that the sets $ \FFF_{t}(F)$, $t=1,...,L_{n,k}-1$ are pairwise disjoint due to \eqref{*VI*a} and 
\eqref{*VI*b}.
We denote by $\mathbf{q}(t)$ the similarity ratio of $\FFF_{t}$.
Since we supposed that $|F|=1$, we also have
\begin{equation}\label{*VI*bb}
\Big (\frac{M_{n}-m_n}{3L_{n,k}} \Big )^{1/\aaa}
\geq
\mathbf{q}(t)
 >
\mathbf{q}_{\min} \Big (\frac{M_{n}-m_n}{3L_{n,k}} \Big )^{1/\aaa}.
\end{equation}

Given $\ddd_{0}>0$, we select $\kkk_{0}>0$, $\rrr_{0}>0$
and $f_{0}^{*}$ according to Proposition \ref{*ko} such that 
\eqref{*eqko}
holds for $f_{0}^{*}$.
Without limiting generality we can suppose that $\rrr_{0}<1$ and $\mathbf{0}\in F$
and $f_{0}^{*}(\mathbf{0})=0$, where $\mathbf{0}$ denotes the origin in $\R^{p}$.       

Put $f_{0}=(1/2)f_{0}^{*}$. From $|F|=1$ and $f_{0}^{*}\in C_{1}^{\aaa}(F)$
it follows that $f_{0}\in C_{1/2}^{\aaa}(F)$ and $|f_{0}(x)|\leq 1/2$ for all $x\in F$.

For $x\in \FFF_{t}(F)$ put
\begin{equation}\label{*VII*b}
f_{n,k}(x):=f_{n}(\mathbf{x}(t))
+\mathbf{q}^{\aaa}(t)\big(f_{0}(\FFF_{t}^{-1}(x))-f_0(\FFF_{t}^{-1}(\mathbf{x}(t)))\big),\quad t=1,...,L_{n,k}-1.
\end{equation}
This way $f_{n,k}$ is well-defined on $F_{n,k}^*=\bigcup_{t=1}^{L_{n,k} -1}\FFF_{t}(F),$ since as we noted, the sets $\FFF_{t}(F)$ are disjoint.

\begin{claim}\label{*cl1}
If $L_{n,k}$ is sufficiently large then
\begin{equation}\label{*VII*aa}
|f_{n,k}(x)-f_{n}(x)|<\frac{1}{n+k}
\text{ for all }x\in F_{n,k}^*.
\end{equation}
\end{claim}

\begin{proof}[Proof of Claim \ref{*cl1}]
Take $x\in F_{n,k}^*$. Then  there exists $t$  such that $x\in \FFF_{t}(F)$.
To obtain  \eqref{*VII*aa} we have the following chain of estimates
$$|f_{n,k}(x)-f_{n}(x)|\leq |f_{n,k}(x)-f_{n,k}(\mathbf{x}(t))|+|f_{n,k}(\mathbf{x}(t))-f_{n}(\mathbf{x}(t))|+
|f_{n}(\mathbf{x}(t))-f_{n}(x)|$$
$$\leq \mathbf{q}^{\aaa}(t)| f_{0}(\FFF_{t}^{-1}(x))-f_0(\FFF_{t}^{-1}(\mathbf{x}(t)))|+0+c_{n}|\mathbf{x}(t)-x|^{\aaa}
$$ 
$$
\leq \mathbf{q}^{\aaa}(t)\frac{1}{2}\Big (  \frac{1}{\mathbf{q}(t)}|x-\mathbf{x}(t)| \Big )^{\aaa}
+c_{n}|\mathbf{x}(t)-x|^{\aaa}\leq \Big (\frac{1}{2}+c_{n}\Big ) |\mathbf{x}(t) -x|^{\aaa}
$$ 
(using \eqref{*VI*b} and choosing a sufficiently large $L_{n,k}$) 
\begin{align}\label{*cl1_vege}
\leq
\frac{M_{n}-m_n}{3L_{n,k}} \Big (\frac{1}{2}+c_{n}\Big )<\frac{1}{n+k}.
\end{align}
This proves Claim \ref{*cl1}.
\end{proof}

\begin{claim}\label{*cl2}
If $L_{n,k}$ is sufficiently large then
\begin{equation}\label{*VII*a}
|f_{n,k}(x)-f_{n,k}(y)|<\frac{1+c_{n}}{2}|x-y|^{\aaa}
\text{ for all }x,y\in F_{n,k}^*.
\end{equation}
\end{claim}

\begin{proof}[Proof of Claim \ref{*cl2}]
Suppose that $x,y\in F_{n,k}^*$.
If  there exists $t$  such that $x,y\in \FFF_{t}(F)$
then
\begin{equation}\label{*IX*a}
\begin{split}
|f_{n,k}(x)-f_{n,k}(y)|
&= \mathbf{q}^{\aaa}(t)|f_{0}(\FFF_{t}^{-1}(x))-
f_{0}(\FFF_{t}^{-1}(y))| \\
&\leq \mathbf{q}^{\aaa}(t)\frac{1}{2}\frac{1}{\mathbf{q}^{\aaa}(t)}|x-y|^{\aaa}\leq \frac{1}{2}
|x-y|^{\aaa}.
\end{split}
\end{equation}

Next suppose that  $x\in\FFF_{t}(F)$ and $y\in \FFF_{t'}(F)$ with $t\not=t'$.
 We separate two subcases. First we suppose that $x$ and $y$
are not too far away. We mean by this  that
\begin{equation}\label{*xynfar}
|x-y|^{1-\aaa}\leq \frac{1-c_{n}}{2}.
\end{equation}

We also need a lower estimate of the distance of $x$ and $y$. We capitalize
on \eqref{*VI*a} and \eqref{*VI*b}
\begin{equation}\label{*X*a}
|x-y| \geq |\mathbf{x}(t)-\mathbf{x}(t')|-|x-\mathbf{x}(t)|-|\mathbf{x}(t')-y|
\end{equation}
$$\geq |\mathbf{x}(t)-\mathbf{x}(t')| \Big ( 1-\Big (   \frac{M_{n}-m_n}{L_{n,k}} \Big )^{\frac{1}{\aaa}}|\mathbf{x}(t)-\mathbf{x}(t')|^{-1} \Big )$$
$$\geq |\mathbf{x}(t)-\mathbf{x}(t')| \Big ( 1-\Big (   \frac{M_{n}-m_n}{L_{n,k}} \Big )^{\frac{1}{\aaa}-1} K_{n} \Big )
\geq \frac{M_{n}-m_n}{L_{n,k}K_n} \Big ( 1-\Big (   \frac{M_{n}-m_n}{L_{n,k}} \Big )^{\frac{1}{\aaa}-1} K_{n} \Big )$$
(supposing that $L_{n,k}$ is sufficiently large)
$$\geq \frac{1}{2} \frac{M_{n}-m_n}{L_{n,k}K_n} .$$
\begin{comment}
The upper estimate is the following:
\begin{equation}\label{*X*b}
|x-y|\
 \leq |\mathbf{x}(t)-\mathbf{x}(t')|+|x-\mathbf{x}(t)|+|\mathbf{x}(t')-y|
\end{equation}
$$\leq |\mathbf{x}(t)-\mathbf{x}(t')|+ 2\Big ( \frac{M_{n}-m_n}{3L_{n,k}}  \Big )^{\frac{1}{\aaa}}
\leq \frac{M_{n}-m_n}{L_{n,k}K_n} \Big ( 1+\Big (   \frac{M_{n}-m_n}{L_{n,k}} \Big )^{\frac{1}{\aaa}-1}K_n \Big )$$
(supposing that $L_{n,k}$ is sufficiently large)
$$\leq 2 \frac{M_{n}-m_n}{L_{n,k}K_n} .$$

Observe that from \eqref{*X*a} it also follows that
\begin{equation}\label{*X*bb}
|\mathbf{x}(t)-\mathbf{x}(t')| \leq |x-y| \Big ( 1-\Big (   \frac{M_{n}-m_n}{L_{n,k}} \Big )^{\frac{1}{\aaa}-1} K_{n} \Big )^{-1}
\end{equation}
$$\leq $$
\end{comment}

For $x$ and $\mathbf{x}(t)$, and for $y$ and $\mathbf{x}(t')$
we use \eqref{*IX*a} to obtain 
$$|f_{n,k}(x)-f_{n,k}(y)|$$
$$\leq |f_{n,k}(x)-f_{n,k}(\mathbf{x}(t))|+|f_{n,k}(\mathbf{x}(t))-f_{n,k}(\mathbf{x}(t'))|+|f_{n,k}(\mathbf{x}(t'))-
f_{n,k}(y)|$$
$$\leq \frac{1}{2}|x-\mathbf{x}(t)|^{\aaa}+|f_{n}(\mathbf{x}(t))-f_{n}(\mathbf{x}(t'))|+\frac{1}{2}|y-\mathbf{x}(t')|^{\aaa}$$
\begin{equation}\label{*diffes}
\leq  \frac{M_{n}-m_n}{3L_{n,k}K_n}+c_{n}|\mathbf{x}(t)-\mathbf{x}(t')|^{\aaa}
\end{equation}
(using \eqref{*X*a})
\begin{equation}\label{*diffesss}
\leq \frac{2}{3}|x-y|+c_{n}|x-y|^{\aaa}\Big (  1- \Big (  \frac{M_{n}-m_n}{L_{n,k}} \Big )^{\frac{1}{\aaa}-1}K_n \Big )^{-\aaa}
\end{equation}
$$\leq |x-y|^{\aaa}\Big (
\frac{2}{3}|x-y|^{1-\aaa}+c_{n}\Big (  1- \Big (  \frac{M_{n}-m_n}{L_{n,k}} \Big )^{\frac{1}{\aaa}-1} K_n \Big )^{-\aaa}
   \Big )$$
 (using \eqref{*xynfar}) 
\begin{equation}\label{*diffess}
\leq |x-y|^{\aaa}\Big (
\frac{1-c_{n}}{3}  +c_{n}\Big (  1- \Big (  \frac{M_{n}-m_n}{L_{n,k}}  \Big )^{\frac{1}{\aaa}-1} K_n \Big  )^{-\aaa}
   \Big )
\end{equation}
(if $L_{n,k}$ is sufficiently large)
$$<\Big (  \frac{1+c_{n}}{2}\Big )|x-y|^{\aaa}.$$
 This took care of the case when \eqref{*xynfar}
holds.

Next we suppose that
\begin{equation}\label{*xynfarr}
|x-y|> \Big (\frac{1-c_{n}}{2}\Big )^{\frac{1}{1-\aaa}}.
\end{equation}
We argue until \eqref{*diffes} as before. 
At this point we can estimate the second term in \eqref{*diffes}
as we did it when we obtained \eqref{*diffesss}.
To estimate the first term 
 using \eqref{*xynfarr} we can choose  
 $L_{n,k}$
sufficiently large such that 
$$
 \frac{M_{n}-m_n}{3L_{n,k}K_n}<\frac{1-c_{n}}{3} \Big (\frac{1-c_{n}}{2}\Big )^{\frac{\aaa}{1-\aaa}}<\frac{1-c_{n}}{3}|x-y|^{\aaa}.
$$
Then
 we can directly jump from 
\eqref{*diffes} to \eqref{*diffess} and then finish the estimate as before.
This completes the proof of Claim \ref{*cl2}.
\end{proof} 

In Claim \ref{*cl2} we have proved that $f_{n,k}$ is $ \frac{1+c_{n}}{2}$-H\"older-$\aaa$ 
on $F_{n,k}^*$. Now, by using Theorem \ref{*Grunb} we extend the definition of $f_{n,k}$
onto $\R^{p}$  such that its extension, still denoted by $f_{n,k}$ is still
a $ \frac{1+c_{n}}{2}$-H\"older-$\aaa$ function.
Put 
\begin{equation}\label{*fnkcsdef}
f_{n,k}^{*}(x)=\min\Big \{ f_{n}(x)+\frac{1}{n+k},
\max\Big \{ f_{n,k}(x), f_{n}(x)-\frac{1}{n+k} \Big \} \Big \}.
\end{equation}

Since $f_n$ is a $ c_{n}$-H\"older-$\aaa$ function and 
$f_{n,k}$ is 
a $ \frac{1+c_{n}}{2}$-H\"older-$\aaa$ function
$f_{n,k}^{*}$ is also a $ \frac{1+c_{n}}{2}$-H\"older-$\aaa$ function. Moreover
 by Claim \ref{*cl1}  
$f_{n,k}^{*}=f_{n,k}$ on $F_{n,k}^*$.
Hence 
\begin{equation}\label{x(t)-ben ugyanaz}
f_{n,k}^{*}(\mathbf{x}(t))=f_{n,k}(\mathbf{x}(t))=f_{n}(\mathbf{x}(t))=\mathbf{p}(t).
\end{equation}
By Proposition \ref{*ko} and by  the choice of $f_0^*$ and $\rho_0$  for any 
$f\in B(f_{0}^{*},\rrr_{0})$
\begin{equation}\label{*XII*a}
\lll\{ r\in f(F):\dimh f^{-1}(r)>\DDD-\ddd_{0} \}\geq \kkk_{0}\lll(f(F)).
\end{equation}
By \eqref{*VII*b}  the graph of $f_{n,k}^{*}|_{\FFF_{t}(F)}$ is an affine copy of 
the graph of $f_{0}^{*}$. In other words, $f_{n,k}^{*}|_{\FFF_{t}(F)}$ is a rescaled  version of $f_{0}^{*}=2f_{0}$ with scaling ratio $\mathbf{q}(t)$ along domain directions and with scaling ratio $\frac12 \mathbf{q}^\alpha(t)$ along the range axis.
This affine transformation also gives a correspondence between $B(f^*_0,\rho_0 )$ and $\{f|_{\Phi_t(F)} : f\in B(f^*_{n,k},\frac12\mathbf{q}^\alpha(t)\rho_0) \}$.
Consequently, by \eqref{*XII*a} for any $f\in B(f_{n,k}^{*},\frac12\mathbf{q}^\alpha(t)\rrr_{0})$
\begin{equation}\label{*XII*n,k}
\lll\{ r\in f(\FFF_{t}(F)):\dimh f^{-1}(r)>\DDD-\ddd_{0} \}
\geq \kkk_{0}\lll(f(\FFF_{t}(F))).
\end{equation}

Set $\ell _{0}=\lll(f_{0}^{*}(F))>0.$
Then using %\eqref{*VII*b} and 
\eqref{*VI*bb}
\begin{equation}\label{*XII*b}
\lll(f_{n,k}^{*}(\FFF_{t}(F)))=\frac{1}{2}\mathbf{q}^{\aaa}( t)\cdot \ell _{0}
\geq \frac{1}{2}\cdot \frac{M_{n}-m_n}{3L_{n,k}} \mathbf{q}_{\min}^{\aaa} \ell _{0}.
\end{equation}
Thus using \eqref{*VI*bb} and $\ell_0\le 1\cdot |F|^\alpha\le 1$,
$$f_{n,k}^{*} (\FFF_{t}(F))
\sse \Big [\mathbf{p}_{n,k}(t)-\frac{1}{2}\cdot \frac{M_{n}-m_n}{3L_{n,k}},
\mathbf{p}_{n,k}(t)+\frac{1}{2}\cdot \frac{M_{n}-m_n}{3L_{n,k}}\Big ].$$

We will select a sufficiently small 
\begin{equation}\label{*suprnk}
0<\rrr_{n,k}< \rrr_{0}\frac{1}{2}\cdot \frac{M_{n}-m_n}{3L_{n,k}}\mathbf{q}_{\min}^{\aaa}
 < \rrr_{0}\frac{1}{2}\cdot \mathbf{q}_t^\alpha
\end{equation}
(the last inequality holds  by \eqref{*VI*bb}). 

Suppose that $f\in B(f_{n,k}^{*},\rrr_{n,k})\sse B(f_{n,k}^{*},\rrr_{0}\mathbf{q}^{\aaa}(t)/2)$.
Then 
$$ f (\FFF_{t}(F))
\sse
\mathbf{I}_{n,k}(t)
:= \Big [\mathbf{p}_{n,k}(t)-\frac{M_{n}-m_n}{3L_{n,k}},
\mathbf{p}_{n,k}(t)+\frac{M_{n}-m_n}{3L_{n,k}}\Big ].$$ 
 By \eqref{*XII*n,k}  we also obtain
\begin{equation}\label{*XIII*a}
\lll\{ r\in f(F)\cap\mathbf{I}_{n,k}(t):\dimh f^{-1}(r)>\DDD-\ddd_{0} \}\geq \kkk_{0}\lll(f(\FFF_{t}(F)))
\end{equation}
$$\geq \kkk_{0}(\lll(f_{n,k}^{*}(\FFF_{t}(F)))-2\rrr_{n,k})$$
(by choosing a sufficiently small $\rrr_{n,k}$ and using \eqref{*XII*b})
$$\geq \frac{\kkk_{0}}{2} \lll(f_{n,k}^{*}(\FFF_{t}(F)))\geq \frac{\kkk_{0}}{2}\cdot \frac{1}{2}\cdot \frac{M_{n}-m_n}{3L_{n,k}} \mathbf{q}_{\min}^{\aaa} \ell _{0}.$$

For $t=1,...,L_{n,k}-1$ the intervals $\mathbf{I}_{n,k}(t)$ are disjoint and equally spaced.

Set $\cag_{k}=\bigcup_{n} B( f_{n,k}^{*} ,\rrr_{n,k})$.  Since $\{ f_{n}\}$ was dense in $C_{1}^{\aaa}(F)$ by \eqref{*fnkcsdef} it is clear that the  sets  $\cag_{k}$  are dense open 
in $C_{1}^{\aaa}(F)$ and hence $\cag=\bigcap_{k}\cag_{k}$ is dense $G_{\ddd}$.

Suppose that $f\in \cag$. 
Then for any $k=1,2,...$  there exists $n(k)$  such that  $f\in B( f_{n(k),k}^{*},\rrr_{n(k),k}).$

Put 
$$\mathbf{H}:=\{ r\in f(F):\dimh f^{-1}(r)<\DDD-\ddd_{0} \}.$$
Proceeding towards a contradiction suppose that
$\lll(\mathbf{H})>0.$

By \eqref{x(t)-ben ugyanaz} we have 
$$
f(F)\supset \Big[m_{ n(k)}+\rrr_{n(k),k},M_{ n(k)}-\frac{M_{n(k)}-m_{n(k)}}{L_{n(k),k}}-\rrr_{n(k),k}\Big], 
$$
and \eqref{*suprnk} implies
$$\rrr_{n(k),k}< \frac{1}{2}\cdot \frac{M_{n(k)}-m_{n(k)}}{L_{n(k),k}}.$$
Then $\mathbf{p}_{n(k),k}(t)\in f(F)$ for $t=1,...,L_{n(k),k}-2$.

By Lebesgue's Density Theorem for every $0<\ggg<1$ for large $k$
 there exists $t\in \{ 2,...,L_{n(k),k}-3 \}$  such that 
 letting $\mathbf{I}^{*}=[\mathbf{p}_{n(k),k}(t-1),\mathbf{p}_{n(k),k}(t+1)]$ we have
\begin{equation}\label{H_suru}
\lll(\mathbf{I}^{*}\cap \mathbf{H})\geq \ggg \lll(\mathbf{I}^{*}). %= 2\ggg \frac{M_{n(k)}-m_{n(k)}}{L_{n(k),k}}.
\end{equation}
 
 On the other hand, $\mathbf{I}_{n(k),k}(t)\sse \mathbf{I}^{*}$.
 Set $$\mathbf{H}^{*}:=\{ r\in f(F)\cap \mathbf{I}^{*}:\dimh f^{-1}(r)\geq \DDD-\ddd_{0} \}.$$
 Using this notation from \eqref{*XIII*a} it follows that
\begin{equation}\label{H^*_suru}
\lll(\mathbf{H}^{*}\cap \mathbf{I}^{*})=\lll(\mathbf{H}^{*})\geq \frac{\kkk_{0}}{12} \frac{M_{n(k)}-m_{n(k)}}{L_{n(k),k}} \mathbf{q}_{\min}^{\aaa}
 \ell _{0}  =  \frac{\kkk_{0}}{12} \mathbf{q}_{\min}^{\aaa}\ell _{0}\frac{\lll(\mathbf{I}^{*})}{{2}}.
\end{equation} 
Since $\mathbf{H}\cap \mathbf{H}^{*}=\ess$, adding \eqref{H_suru} to \eqref{H^*_suru} we obtain
\begin{align*}
\lambda(\mathbf{I}^*) &\ge \gamma\lambda(\mathbf{I^*})+\frac{\kkk_{0}}{12} \mathbf{q}_{\min}^{\aaa}\ell _{0}\frac{\lll(\mathbf{I}^{*})}{2} \\
1 &\ge \gamma+\frac{\kkk_{0}}{12} \mathbf{q}_{\min}^{\aaa}\ell _{0}\frac{1}{2} \\
1-\gamma &\ge \frac{\kkk_{0}}{12} \mathbf{q}_{\min}^{\aaa}\ell _{0}\frac{1}{2}.
\end{align*}
This yields a contradiction, as $\gamma$ can be chosen arbitrarily close to $1$.
\end{proof}

%\begin{thebibliography}{99}

%\bibitem{[GrunbHolderext]} 
%%{\sc F. Gr\"unbaum, E. H. Zarantonello}, {\em On the extension of uniformly continuous mappings}, Michigan Math. 
%J. 15 (1968), pp.~65--74.

%\end{thebibliography}

\bibliographystyle{amsplain} 
\bibliography{sier}

\end{document}